\documentclass[reqno,twoside]{article}
\usepackage{amsfonts}
\usepackage{amsmath,amssymb}
\usepackage{epic}
\usepackage{pstricks}
\usepackage{graphicx,xcolor}
\usepackage{array}

\usepackage{subcaption}

\usepackage{float}
 
\usepackage{a4wide,amsmath,amssymb,latexsym,amsthm}
\usepackage{eucal}
\usepackage{graphicx}

\newtheorem{theorem}{Theorem}[section]

\newtheorem{remark}[theorem]{Remark}
\newtheorem{lemma}[theorem]{Lemma}
 
\newtheorem{definition}[theorem]{Definition}

\numberwithin{equation}{section}

\graphicspath{{./Imgs/}}

\newcommand{\R}{\mathbb R}
 
\newcommand{\N}{\mathbb N}

\newcommand{\be}{\begin{equation}}
\newcommand{\ee}{\end{equation}}
\newcommand{\ba}{\begin{eqnarray}}
\newcommand{\ea}{\end{eqnarray}}
\newcommand{\beq}{\begin{equation}}
\newcommand{\eeq}{\end{equation}}

\usepackage{color}
\definecolor{red}{rgb}{0,0,0}

\usepackage[textsize=small]{todonotes}

\numberwithin{equation}{section}

\def\Omc{\R\setminus (-1,1)}

\def\RR{{\mathbb{R}}}
\def\NN{{\mathbb{N}}}

\def\Om{(-1,1)}

\def\bbOm{\Omega}
\def\bOm{\Om}

\usepackage{color}
\usepackage{stmaryrd}

\newcommand{\flap}{(-\Delta)^s}

\usepackage[colorlinks=true, linkcolor=blue, citecolor=blue, urlcolor=blue]{hyperref}

\begin{document}

\title{A control-based spatial source reconstruction in fractional heat equations}

\author{Galina Garc\'ia\thanks{Universidad de Santiago de Chile, Departamento de
Matem\'atica y Ciencia de la Computaci\'on, Facultad de Ciencia, Casilla 307-Correo 2,
Santiago, Chile. \texttt{galina.garcia@usach.cl}, \texttt{joaquin.vidal.p@usach.cl}} \and Joaqu\'in Vidal\footnotemark[1]  \and
Sebasti\'an Zamorano\thanks{Faculty of Engineering, University of Deusto, Av. Universidades 24, 48007, Bilbao, Basque Country, Spain. \texttt{sebastian.zamorano@deusto.es}}}
 
 

\date{}

\maketitle

\begin{abstract}
In this study, an inverse source problem for a nonlocal heat equation governed by the fractional Laplacian is analyzed. The main objective is to recover the spatial component of a separable source term $F(x,t)=f(x)\sigma(t)$ from partial measurements of the solution $u$ and its time derivative $u_t$ on a subset of the spatial domain. The proposed approach relies on the null controllability of the fractional heat equation when the diffusion order satisfies $s\in(1/2,1)$, linking the controllability theory with inverse problem techniques.
By combining spectral analysis, Volterra integral equations, and control-based arguments, we derived an explicit reconstruction formula for the Fourier coefficients of the unknown function $f$. This formula provides a rigorous and computationally feasible framework for recovering spatially distributed sources with limited observational data. A numerical scheme based on finite element discretization of the fractional Laplacian and penalized Hilbert Uniqueness Method (HUM) was developed to compute the required controls and implement the reconstruction algorithm.
Several numerical experiments confirmed the effectiveness, accuracy, and stability of the proposed method, even when the available data were restricted to a subdomain. The results highlight the potential of control-theoretic tools for addressing inverse problems in fractional diffusion models and open perspectives for applications in imaging and anomalous transport phenomena.

\end{abstract}


\section{Introduction}
\subsection{The inverse problem}
Inverse problems for partial differential equations (PDEs) have garnered significant attention owing to their wide-range of applications in engineering, physics, and medical imaging. Among these, the inverse source problem, where the goal is to reconstruct an unknown source term from partial measurements of the state of the system, is of particular interest. This problem becomes particularly challenging when dealing with nonlocal operators, such as the fractional Laplacian, which models anomalous diffusion processes encountered in heterogeneous materials, biology, and finance.

In this study, we addressed the inverse source problem for the fractional heat equation, focusing on the reconstruction of the spatial component of a source term from limited observations of the state of the system and its time derivative over a sub-domain. Namely, let $\omega \subset (-1,1)$ be an arbitrary nonempty subset, and let $T > 0$. Let us consider the nonlocal heat equation as
\begin{align}\label{HE-EX}
\begin{cases}
\partial_t u + (-\partial_x^2)^{s} u = F(x,t) & \text{in } (-1,1) \times (0,T),\\
u = 0 & \text{in } (\mathbb{R} \setminus (-1,1)) \times (0,T), \\
u(\cdot,0) = 0 & \text{in } (-1,1).
\end{cases}
\end{align}
In equation \eqref{HE-EX}, $u = u(x,t)$ represents the state, $T > 0$ and $0 < s < 1$ are real parameters, and $F(x,t) = f(x)\sigma(t)$ is the source term with a spatial component \( f(x) \) and a temporal factor \( \sigma(t) \). The operator $(-\partial_x^2)^s$ denotes the fractional Laplacian, which is formally defined by the singular integral
\begin{align*}
(-\partial_x^2)^s u(x) = C_{s} \, \text{P.V.} \int_{\mathbb{R}} \frac{u(x) - u(y)}{|x - y|^{1 + 2s}} \, dy, \quad x \in \mathbb{R},
\end{align*}
where $C_{s}$ is a normalization constant that depends only on $s$. In Section \ref{sec-2}, we provide a precise definition. For simplicity, the initial condition is set to zero.

The main goal of this study is to investigate an inverse source problem for the nonlocal heat equation \eqref{HE-EX} assuming that $F$ can be written as $F(x,t)=f(x)\sigma(t)$ with a known temporal factor $\sigma\in W^{1,\infty}(0,T)$ and $f\in L^2(-1,1)$ as an unknown spatial component. The aim of this study is to recover the unknown spatial source term $f$ based on the observations of $u$ and $u_t$ within a subset $\omega\subset (-1,1)$ over the time interval $(0,T)$. Moreover, we determine a reconstruction formula for the Fourier coefficients of $f$.

Based on the pioneering work of \cite{yamamoto1995stability} concerning an inverse source hyperbolic problem and in \cite{garcia2017source,garcia2013heat} for Stokes and heat equations, respectively, our work builds on the interplay between control theory, spectral analysis, and Volterra integral equations to derive a robust reconstruction formula. The key contributions of this study are as follows.

\begin{enumerate}
\item We establish a novel reconstruction formula for the Fourier coefficients of the unknown spatial source term. This formula takes advantage of the null controllability property of the fractional heat equation, which holds when fractional order $s\in(1/2,1)$. Our approach relies on spectral decomposition and the solution of the associated Volterra integral equations, enabling recovery of the source under limited measurement data.

\item This methodology combines the spectral properties of the fractional Laplacian with control-based techniques. By exploiting the eigenvalues and eigenfunctions of the fractional Laplacian, we reduce the problem of solving a family of control problems and integral equations, which are subsequently implemented numerically.

\item We provide a comprehensive numerical scheme to approximate the eigenvalues and eigenfunctions of the fractional Laplacian, solve optimal control problems, and implement the reconstruction formula. Numerical experiments demonstrated the accuracy and stability of the proposed approach, even for sources with high-frequency components or discontinuities. Specifically, we employed the penalized Hilbert Uniqueness Method (HUM), which allows us to compute the control inputs while minimizing the associated computational cost. The resulting numerical scheme not only achieves high accuracy in reconstructing the Fourier coefficients of $f(x)$ but also maintains computational efficiency through a carefully designed balance between the number of eigenfunctions used and the overall cost.

\item Our analysis extends existing results on inverse problems for fractional PDEs by incorporating control-theoretic tools. The derived reconstruction formula was validated theoretically and numerically, offering a systematic method to handle inverse problems in nonlocal settings.
\end{enumerate}

This study contributes to the growing body of research on inverse problems for fractional PDEs by providing a practical and theoretically sound framework for source reconstructions. The results have potential applications in imaging, material science, and other fields where nonlocal diffusion processes play a central role.

\subsection{State of the art}
Nonlocal operators, particularly the fractional Laplacian, have become indispensable tools for modelling phenomena in which traditional local PDEs fail, such as anomalous diffusion, turbulence, and wave propagation in heterogeneous media (e.g., \cite{HAntil_SBartels_2017a, BBC, Val, SV2} and references therein). Their ability to capture long-range interactions makes them ideal for application in physics, finance, and imaging. However, the nonlocality and singular kernels of these operators introduce significant challenges in terms of analysis and numerics, particularly for inverse problems.

Inverse problems for fractional PDEs, such as reconstructing unknown sources or coefficients from partial measurements, have been increasingly studied owing to their relevance in thermal imaging and material science \cite{MR4065621,MR4718707,MR4237942,MR4083776,MR4078233,HELIN20207498}. Although progress has been made in source identification for fractional diffusion equations \cite{MR4185290,zhang2011inverse}, reconstructing the spatial component of a source term under limited observations remains underexplored. This gap is particularly pronounced for the fractional heat equation, where nonlocality complicates both the theoretical analysis and numerical approximation.

Our study addresses this challenge by combining spectral methods, control theory, and Volterra integral equations to derive a rigorous reconstruction formula. By leveraging the null controllability of the fractional heat equation, we provide a computationally feasible framework for recovering spatial sources from partial data, thereby bridging the theoretical advances with practical applications.

\subsection{Outline of the article}
The remainder of this paper is organized as follows. In Section \ref{sec-2}, we introduce the function spaces required for our analysis and present the intermediate results essential for the main proofs. In Section \ref{sec-3}, we derive and prove the reconstruction formula for the inverse problem. In Section \ref{sec-4}, we discuss an optimal control approach related to the reconstruction methodology. Finally, in Section \ref{sec-5}, we detail the numerical scheme used in this study, and Section \ref{sec-6} presents numerical experiments that illustrate the effectiveness of our approach. Finally, Section \ref{sec-7}, presents our concluding remarks.


\section{Preliminary results}\label{sec-2}
In this section, we introduce the necessary notations and recall some known results that are essential for stating and proving our main results.

\subsection{Functional setting}
First, we define the fractional-order Sobolev space.  For $0<s<1$ and $\Omega\subset\RR$ an arbitrary open set,  we let
\begin{align*}
H^{s}(\Omega):=\left\{u\in L^2(\Omega):\;\int_\Omega\int_\Omega\frac{|u(x)-u(y)|^2}{|x-y|^{1+2s}}\;dxdy<\infty\right\},
\end{align*}
and we endow it with the norm defined by
\begin{align*}
\|u\|_{H^{s}(\Omega)}:=\left(\int_\Omega|u(x)|^2\;dx+\int_\Omega\int_\Omega\frac{|u(x)-u(y)|^2}{|x-y|^{1+2s}}\;dxdy\right)^{\frac 12}.
\end{align*}
We set
\begin{align*}
H_0^{s}(\bbOm):=\Big\{u\in H^{s}(\R):\;u=0\;\mbox{ a.e. in }\;\R\setminus \Omega\Big\}
=\Big\{u\in H^{s}(\R),\;\;\operatorname{supp}[u]\subset\;\bbOm\Big\}.
\end{align*}
If $\Omega=\RR$ in the above definition, then we mean that $H_0^{s}(\overline{\RR})=H^s(\RR)$.

We shall denote by $H^{-s}(\bbOm)$ the dual of $H_0^{s}(\bbOm)$ with respect to the pivot space $L^2(\Omega)$, that is, $H^{-s}(\bbOm)=(H_0^{s}(\bbOm))^\star$.  The following continuous embedding holds
\begin{align*}
H_0^s(\overline{\Omega})\hookrightarrow L^2(\Omega)\hookrightarrow H^{-s}(\overline{\Omega}).
\end{align*}
For more information on fractional Sobolev spaces, refer to \cite{NPV,War} and references therein.

\subsection{Fractional Laplacian}
Next, we provide a rigorous definition of the fractional Laplace operator. Let 
\begin{align*}
\mathcal L_s^{1}(\R):=\left\{u:\R\to\R\;\mbox{ measurable},\; \int_{\R}\frac{|u(x)|}{(1+|x|)^{1+2s}}\;dx<\infty\right\}.
\end{align*}
For $u\in \mathcal L_s^{1}(\R)$ and $\varepsilon>0$ we set
\begin{align*}
(-\partial_x^2)_\varepsilon^s u(x):= C_{s}\int_{\{y\in\R:\;|x-y|>\varepsilon\}}\frac{u(x)-u(y)}{|x-y|^{1+2s}}\;dy,\;\;x\in\R,
\end{align*}
where $C_{s}$ is a normalization constant given by
\begin{align}\label{CNs}
C_{s}:=\frac{s2^{2s}\Gamma\left(\frac{2s+1}{2}\right)}{\pi^{\frac
12}\Gamma(1-s)}.
\end{align}
The {\em fractional Laplacian}  $(-\partial_x^2)^s$ is then defined for $u\in \mathcal L_s^{1}(\R)$ by the singular integral:
\begin{align}\label{fl_def}
(-\partial_x^2)^su(x):=C_{s}\,\mbox{P.V.}\int_{\R}\frac{u(x)-u(y)}{|x-y|^{1+2s}}\;dy=\lim_{\varepsilon\downarrow 0}(-\partial_x^2)_\varepsilon^s u(x),\;\qquad x\in\R,
\end{align}
provided that the limit above exists.
We notice that if $u\in \mathcal L_s^{1}(\R)$, then $ v:=(-\partial_x^2)_\varepsilon^s u$ exists for every $\varepsilon>0$, $v$ is also continuous at the continuity points of  $u$.  
For more details on the fractional Laplace operator, see \cite{Caf3,NPV,GW-CPDE,War} and references therein.

Now, we consider the realization of $(-\partial_x^2)^s$ in $L^2(-1,1)$ with the zero-Dirichlet condition $u=0$ in $\Omc$. More precisely, we consider the closed and bilinear form $\mathcal F:H_0^{s}(\bOm)\times H_0^{s}(\bOm)\to\RR$ given by
\begin{align*}
\mathcal F(u,v):=\frac{C_{s}}{2}\int_{\R}\int_{\R}\frac{(u(x)-u(y))(v(x)-v(y))}{|x-y|^{1+2s}}\;dxdy,\;\qquad u,v\in H_0^{s}(\bOm).
\end{align*}
Let $(-\partial_x^2)_D^s$ be the selfadjoint operator in $L^2(-1,1)$ associated with $\mathcal F$ in the sense that
\begin{equation*}
\begin{cases}
D((-\partial_x^2)_D^s)=\Big\{u\in H_0^{s}(\bOm),\;\exists\;f\in L^2(-1,1),\;\mathcal F(u,v)=(f,v)_{L^2(-1,1)}\;\forall\;v\in H_0^{s}(\bOm)\Big\},\\
(-\partial_x^2)_D^su=f.
\end{cases}
\end{equation*}
It is easy to see that
\begin{equation}\label{DLO}
D((-\partial_x^2)_D^s)=\left\{u\in H_0^{s}(\bOm),\; (-\partial_x^2)^su\in L^2(-1,1)\right\},\;\;\;
(-\partial_x^2)_D^su=((-\partial_x^2)^su)|_{(-1,1)}.
\end{equation}
Then, $(-\partial_x^2)_D^s$ is the realization of $(-\partial_x^2)^s$ in $L^2(-1,1)$ with the condition $u=0$ in $\Omc$. It has been shown in \cite{SV2} that $(-\partial_x^2)_D^s$ has a compact resolvent and its eigenvalues form a non-decreasing sequence of real numbers $0<\lambda_1\le\lambda_2\le\cdots\le\lambda_n\le\cdots$ satisfying $\lim_{n\to\infty}\lambda_n=\infty$.  In addition, if $\frac 12\le s<1$, the eigenvalues have a finite multiplicity.
Let $\{\varphi_n\}_{n\in\NN}$ be the orthonormal basis of the eigenfunctions associated with eigenvalues $\{\lambda_n\}_{n\in\NN}$. Then, $\varphi_n\in D((-\partial_x^2)_D^s)$ for every $n\in\NN$,  $\{\varphi_n\}_{n\in\NN}$ is total in $L^2(-1,1)$ and satisfies 
\begin{equation}\label{ei-val-pro}
\begin{cases}
(-\partial_x^2)^s\varphi_n=\lambda_n\varphi_n\;\;&\mbox{ in }\;(-1,1),\\
\varphi_n=0\;&\mbox{ in }\;\Omc.
\end{cases}
\end{equation}

Let us mention the fact that the eigenvalues $\{\lambda_n\}_{n\in\N}$ of the realization of $(-\partial_x^2)^s$ in $L^2(-1,1)$ with the zero Dirichlet exterior condition  satisfy the following asymptotic formula:
\begin{align}\label{lam}
\lambda_n=\left(\frac{n\pi}{2}-\frac{(2-2s)\pi}{8}\right)^{2s}+O\left(\frac{1}{n}\right)\;\text{ as }\, n\to\infty,
\end{align}
an important result that was proved in \cite{kwasnicki2012eigenvalues}. This property plays a crucial role in both theoretical and numerical results. Note that, in contrast to the local Laplace operator, this is the only known behavior of the eigenvalues of $(-\partial_x^2)_D^s$.

Next, for $u\in H^{s}(\R)$ we introduce the {\em nonlocal normal derivative $\mathcal N_s$} given by 
\begin{align}\label{NLND}
\mathcal N_su(x):=C_{s}\int_{-1}^1\frac{u(x)-u(y)}{|x-y|^{1+2s}}\;dy,\;\;\;x\in\R \setminus\bOm,
\end{align}
where $C_{s}$ is a constant given by \eqref{CNs}.
To conclude these preliminaries regarding the fractional Laplace operator, the following integration by parts formula is presented.

\begin{lemma}
Let $u\in H_0^s(\bOm)$ be such that $(-\partial_x^2)^s u\in L^2(-1,1)$. Then for every $v\in H^s(\RR)$, the following identity holds
\begin{align}\label{Int-Part}
\frac{C_{s}}{2}\int_{\RR}\int_{\RR}\frac{(u(x)-u(y))(v(x)-v(y))}{|x-y|^{1+2s}}\;dxdy=\int_{-1}^1v(x)(-\partial_x^2)^su(x)\;dx+\int_{\Omc}v(x)\mathcal N_su(x)\;dx,
\end{align}
holds.
\end{lemma}
We refer to \cite[Lemma 3.3]{DRV} (see also \cite[Proposition 3.7]{War-ACE}) for the proof and more details.

\subsection{Nonlocal heat equation}
In this section, we present the well-known results regarding the existence, uniqueness, and regularity of the solutions to equation \eqref{HE-EX}. In addition, we recall the null controllability result for the nonlocal heat equation.

Throughout this section and the remainder of the article, $\{\varphi_n\}_{n\in\NN}$ denotes the orthonormal basis of the eigenfunctions of the operator $(-\partial_x^2)_D^s$ associated with the eigenvalues $\{\lambda_n\}_{n\in\NN}$. Furthermore, for a given $u\in L^2(-1,1)$ and $n\in\NN$, we let $u_n:=(u,\varphi_n)_{L^2(-1,1)}$ and for a given set $E\subseteq \RR^N$ ($N\ge 1)$, we  	denote by $(\cdot,\cdot)_{L^2(E)}$ the scalar product in $L^2(E)$. Moreover,  $\mathcal{D}(E)$ denotes the space of all continuously infinitely differentiable functions with compact support in $E$.

For completeness, we introduce the notion of a solution for the problem
\begin{align}\label{IPHE}
\begin{cases}
\partial_t u + (-\partial_x^2)^{s} u = F(x,t) & \mbox{ in }\; (-1,1)\times(0,T),\\
u=0 &\mbox{ in }\; (\Omc)\times (0,T), \\
u(\cdot,0) = 0&\mbox{ in }\; (-1,1).
\end{cases}
\end{align}

\begin{definition}\label{D1}
Let $F\in L^2(0,T; H^{-s}(-1,1))$. We say that $u\in  L^2(0,T; H_0^s(-1,1))$, with $u_t\in L^2(0,T; H^{-s}(-1,1))$, is a finite-energy solution of the nonlocal heat equation  \eqref{IPHE} if the following identity holds:
\begin{multline}
\int_0^T\int_{-1}^1 u_t w dxdt +\frac{C_s}{2}\int_0^T\int_{-1}^1\int_{-1}^1 \frac{(u(x)-u(y))(w(x)-w(y))}{|x-y|^{1+2s}}dxdydt\\ = \int_0^T \langle F, w\rangle_{H^{-s}(-1,1), H_0^s(-1,1)} dt,
\end{multline}
for any $w\in L^2(0;T; H_0^s(-1,1))$, where $\langle \cdot, \cdot \rangle_{H^{-s}(-1,1), H_0^s(-1,1)}$ denotes the duality pairing between $H^{-s}((-1,1))$ and $H_0^s(-1,1)$.
\end{definition}

\begin{remark}\label{r1}
\begin{enumerate}
\item According to \cite[Theorem 10]{leonori2015basic}, for any  $F\in L^2(0,T;H^{-s}(-1,1))$ the system \eqref{IPHE} admits a unique weak solution according to Definition \ref{D1}.

\item Moreover, if $u\in L^2(0,T; H_0^s((-1,1)))$ and $u_t\in L^2(0,T; H^{-s}((-1,1)))$, then $u\in C([0,T]; L^2(-1,1))$. Thus the identity $u(\cdot,0)=0$ makes sense in $L^2(-1,1)$.

\item If $F\in L^2(\omega\times(0,T))$ and $u_0=0$, then \cite[Theorem 1.5]{BWZ1} shows that
\begin{align}\label{regularity}
u\in L^2(0,T; H_{loc}^{2s}(-1,1))\cap L^{\infty}(0,T; H_0^s((-1,1)))\quad \text{and}\quad u_t\in L^2((-1,1)\times (0,T)).
\end{align}

\item Finally, if $F(x,t)=f(x)\sigma(t)$, it is immediately shown that the solution of \eqref{IPHE} with non-zero initial data $u(x,0)=u_0(x)$, can be expressed based on the eigenfunctions $\{\varphi_n\}_{n\in\NN}$ of the operator $(-\partial _x^2)_D^s$ as follows:
\begin{align}\label{series1}
u(x,t) =\sum_{k\geq 1}u_{0,n}e^{-\lambda_k t}\varphi_k(x)+ \sum_{k\geq 1}\left( f_{k} \int_0^t e^{-\lambda_k(t-s)}\sigma(s)ds \right)\varphi_k(x).
\end{align}
\end{enumerate}
\end{remark}

Another important result for the nonlocal heat equation on which our study is based is the following null controllability of the nonlocal heat equation. The proof of this theorem employs spectral analysis techniques based on \cite{kwasnicki2012eigenvalues} and can be found in \cite{biccaricontrollability}.
\begin{theorem}
Given $\tau \in (0,T]$, $\omega\subset \Omega$ with nonempty interior, and $\varphi_0\in L^2(-1,1)$, there exists a control function $h^{(\tau)}=h^{(\tau)}(\varphi_0)\in L^2(0,\tau;L^2(\omega))$ such that the corresponding solution $\phi$ of the problem
\begin{align}\label{null-control}
\begin{cases}
-\partial_t \phi + (-\partial_x^2)^{s} \phi = h^{(\tau)}\chi_{\omega\times(0,\tau)} & \text{ in }\; (-1,1)\times(0,\tau),\\
\phi=0 &\mbox{ in }\; (\Omc)\times (0,\tau), \\
\phi(\cdot,\tau) = \varphi_0&\mbox{ in }\; (-1,1),
\end{cases}
\end{align}
satisfies
\begin{align}\label{null-control-2}
\phi(\cdot, 0) = 0\quad \text{in }\Omega,
\end{align}
if and only if $s>\frac 12$. Moreover, there exists a constant $C>0$ such that the following inequality holds
\begin{align}\label{obs-ine}
\|h^{(\tau)}\|_{L^2(0,\tau; L^2(\omega))}\leq C\|\varphi_0\|_{L^2(-1,1)}.
\end{align}
\end{theorem}


\subsection{Volterra integral equations}

To conclude this preliminary section, we recall some well-known facts regarding the Volterra integral equations of the first and second kind. The following results, as well as more details on these equations, can be found in \cite{tricomi1985integral}.

\begin{lemma}\label{Lemma1}
For $0<t_1\leq T$, $\sigma\in W^{1,\infty}(0,t_1)$, and any $\eta \in L^2((0,t_1)\times (-1,1))$, a solution $\theta^{(t_1)}\in H^1(0,t_1; L^2(-1,1))$ of the following  Volterra integral equation of the second kind exists:
\begin{align}\label{Volterra-1}
\begin{cases}
\eta(x,t)=\sigma(0)\theta_t^{(t_1)}(x,t)+\displaystyle\int_t^{t_1}\Big(\sigma(s-t)\theta^{(t_1)}(x,s)+\sigma'(s-t)\theta_t^{(t_1)}(x,s)\Big)ds,&\\
\theta^{(t_1)}(x,t_1)=0.
\end{cases}
\end{align}
Besides, there exists a constant $C>0$ depending on $\|\sigma\|_{W^{1,\infty}(0,t_1)}$ such that
\begin{align}\label{Volterra-2}
\|\theta^{(t_1)}\|_{H^1(0,t_1;L^2(-1,1))}\leq C\|\eta\|_{L^2((0,t_1)\times(-1,1))}.
\end{align}
\end{lemma}

\begin{lemma}\label{Lemma2}
Let $\sigma\in W^{1,\infty}(-1,1)$. We define the operator $K: L^2((0,T)\times(-1,1))\to H^1(0,T; L^2(-1,1))$ by
\begin{align}\label{Volterra-3}
(Kv)(x,t):=\int_0^t \sigma(\tau)v(x,t-\tau)d\tau.
\end{align}
Then, there exists a constant $C>0$ depending on $(-1,1), T$, and $\|\sigma\|_{W^{1,\infty}(0,t_1)}$, for $0<t_1\leq T$, such that 
\begin{align}\label{Volterra-4}
C\|Kv\|_{H^1(0,T; L^2(-1,1))}\leq \|v\|_{L^2((0,T)\times(-1,1))}\leq \|Kv\|_{H^1(0,T; L^2(-1,1))}.
\end{align}
Moreover, the adjoint operator $K^{*}: H^1(0,T; L^2(-1,1))\to L^2((0,T)\times(-1,1)) $ is given by 
\begin{align}\label{Volterra-5}
(K^{*}\theta)(x,t)=\sigma(0)\theta_t(x,t)+\int_t^T \Big(\sigma(\tau-t)\theta(x,\tau)+\sigma'(\tau-t)\theta_t(x,\tau)\Big)d\tau.
\end{align}
\end{lemma}

Finally, let us consider the following nonlocal heat equation
\begin{align}\label{HE1}
\begin{cases}
\partial_t w + (-\partial_x^2)^{s} w = 0 & \mbox{ in }\; (-1,1)\times(0,T),\\
w=0 &\mbox{ in }\; (\Omc)\times (0,T), \\
w(\cdot,0) = f(x)&\mbox{ in }\; (-1,1).
\end{cases}
\end{align}
Using classical Duhamel's principle, but in the context of the fractional Laplacian, we obtain the following relation for the solution of \eqref{IPHE}, with $F(x,t)=f(x)\sigma(t)$, and \eqref{HE1}.

\begin{lemma}\label{Lemma3}
Let $f\in L^2(-1,1)$, $\sigma\in W^{1,\infty}(0,T)$, and $u,w$ be the unique weak solutions of \eqref{IPHE}, with $F(x,t)=f(x)\sigma(t)$, and \eqref{HE1}, respectively. Then, the solution $u$ satisfies
\begin{align}\label{Volterra-6}
u(x,t)=(Kw)(x,t),
\end{align}
where the operator $K:L^2((0,T)\times(-1,1))\to H^1(0,T; L^2(-1,1))$ is given by \eqref{Volterra-3}.
\end{lemma}
Let us remark that this result is reasonable owing to the regularity of the finite-energy solution of the fractional heat equation. Without extra regularity, that is, $u_t\in L^2((0,T)\times(-1,1))$, we do not address the methodology used in this study.

\section{Main results}\label{sec-3}
In this section, we state and prove the main results of the reconstruction formula for an unknown source term. Our methodology is strongly based on the null controllability of the nonlocal heat equation, the spectral properties of the fractional Laplacian operator, and the Volterra integral equations introduced in the previous section.

Our first main result establishes a reconstruction formula for the source term $f$ from the local observations of state $u$ and its time derivative $u_t$ in subdomain $\omega\times (0,T)$. The formula is given in terms of the Fourier coefficients $f_n=(f,\varphi_n)_{L^2(-1,1)}$, where $\{\varphi_n\}_{n\in\N}$ is the orthonormal basis of the eigenfunctions associated with the eigenvalues $\{\lambda_n\}_{n\in\NN}$ of the fractional Laplace operator.

\begin{theorem}\label{th1}
Let $\sigma\in W^{1,\infty}(0,T)$ with $\sigma(T)\neq 0$, $f\in L^2(-1,1)$, $s\in (1/2,1)$, and let $\{\lambda_n,\varphi_n\}_{n\in\N}$ be the eigenvalues and orthonormal eigenvectors of the fractional Laplace operator in $(-1,1)$ with homogeneous Dirichlet exterior conditions. Given $n\in\N$, for each $0<\tau\leq T$, let $h^{(\tau)}$ be a null control associated with problem \eqref{null-control} with initial data $\varphi_n$, extended by zero in $(\tau,T]$. Let $\theta^{(\tau)}$ be the solution of the Volterra integral equation \eqref{Volterra-1} for $\eta=h^{(\tau)}$ also extended by zero in $(\tau,T]$. Then, the following reconstruction formula holds
\begin{multline}\label{th1.1}
( f,\varphi_n)_{L^2(-1,1)}=
\sigma(T)^{-1}\left(-\sigma(0)\langle u,\theta^{(T)}\rangle_{H^1(0,T;L^2(\omega))}
-\int_0^T\sigma'(T-\tau)\langle u,\theta^{(\tau)}\rangle_{H^1(0,T;L^2(\omega))}d \tau\right.\\
\left.-\lambda_n\int_0^T\sigma(T-\tau)\langle u,\theta^{(\tau)}\rangle_{H^1(0,T;L^2(\omega))}d \tau\right).
\end{multline} 
\end{theorem}

\begin{proof}
From Lemmas \ref{Lemma2} and \ref{Lemma3},  we know that the unique finite-energy solution of \eqref{IPHE} is related to the unique weak solution of
\begin{align}\label{ec2}
\begin{cases}
\partial_t w + (-\partial_x^2)^{s} w = 0 & \mbox{ in }\; (-1,1)\times(0,T),\\
w=0 &\mbox{ in }\; (\Omc)\times (0,T), \\
w(\cdot,0) = f(x)&\mbox{ in }\; (-1,1),
\end{cases}
\end{align}
through the operator $K:L^2((0,T)\times(-1,1))\to H^1(0,T; L^2(-1,1))$
\begin{align}\label{ec1}
u(x,t)=(Kw)(x,t)=\int_0^t \sigma(t-\tau)w(x,\tau)d\tau.
\end{align}

Taking time derivative to \eqref{ec1}, we obtain
\begin{align}\label{ec3}
u_t(x,t)=\sigma(0)w(x,t)+\int_0^t \sigma'(t-\tau)w(x,\tau)d\tau.
\end{align}
Because $u$ is the unique weak solution of \eqref{HE-EX}, by replacing \eqref{ec3} in \eqref{HE-EX}, we obtain
\begin{align}\label{ec4}
\begin{cases}
 \sigma(0)w(x,t)+\displaystyle\int_0^t \sigma'(t-\tau)w(x,\tau)d\tau+ (-\partial_x^2)^{s} u = f(x)\sigma(t) & \mbox{ in }\; (-1,1)\times(0,T),\\
u=0 &\mbox{ in }\; (\Omc)\times (0,T), \\
u(\cdot,0) = 0&\mbox{ in }\; (-1,1).
\end{cases}
\end{align}

Therefore, by evaluating in $t=T$ the first equation in \eqref{ec4}, multiplying by $\varphi_k$, with $k\in\NN$ and integrating over $(-1,1)$, we deduce 
\begin{multline}\label{ec5}
\sigma(0)(w(\cdot,T),\varphi_k)_{L^2(-1,1)}+\int_0^T\sigma'(T-\tau)(w(\cdot,\tau),\varphi_k)_{L^2(-1,1)}d\tau\\ + ((-\partial_x^2)^s u(\cdot,T),\varphi_k)_{L^2(-1,1)}=\sigma(T)(f,\varphi_k)_{L^2(-1,1)}.
\end{multline}
Then, it follows that the Fourier coefficient of $f$ is given by
\begin{multline}\label{ec6}
(f,\varphi_k)_{L^2(-1,1)}=\frac{\sigma(0)}{\sigma(T)}(w(\cdot,T),\varphi_k)_{L^2(-1,1)}+\frac{1}{\sigma(T)}\int_0^T \sigma'(T-\tau)(w(\cdot,\tau),\varphi_k)_{L^2(-1,1)}d\tau\\+\frac{1}{\sigma(T)}((-\partial_x^2)^s u(\cdot,T),\varphi_k)_{L^2(-1,1)}.
\end{multline}

We observe that by using the integration by parts formula \eqref{Int-Part}, we obtain:
\begin{align*}
\frac{1}{\sigma(T)}((-\partial_x^2)^s u(\cdot,T),\varphi_k)_{L^2(-1,1)}=\frac{1}{\sigma(T)}(u(\cdot,T),(-\partial_x^2)^s\varphi_k)_{L^2(-1,1)}.
\end{align*}
Because $\{\varphi_k\}_{k\in\NN}$ satisfies the eigenvalue problem \eqref{ei-val-pro} and from the definition of operator $K$, we obtain
\begin{align}\label{ec7}
\frac{1}{\sigma(T)}((-\partial_x^2)^s u(\cdot,T),\varphi_k)_{L^2(-1,1)}=\frac{\lambda_k}{\sigma(T)}\int_0^T\sigma(T-\tau)(w(\cdot,\tau),\varphi_k)_{L^2(-1,1)}d\tau.
\end{align}

Next, we must estimate the terms $(w(\cdot,\tau),\varphi_k)_{L^2(-1,1)}$ for every $k\in\NN$ and $\tau\in(0,T)$. Let $\tau\in (0,T)$ and, for each $k\in\NN$, let $\phi$ be the unique solution of \eqref{null-control} over $(-1,1)\times(0,\tau)$ with initial datum $\varphi_k$ and control $h^{(\tau)} $, where both $\phi$ and $h^{(\tau)}$ are zero-extended in $(\tau,T]$. Using $\phi$ as a test function in the definition of a weak solution to the problem \eqref{ec2}, we have
\begin{align}\label{ec8}
\int_0^{\tau}\int_{-1}^1 w_t \phi dxdt +\frac{C_s}{2}\int_0^{\tau}\int_{-1}^1\int_{-1}^1 \frac{(w(x)-w(y))(\phi(x)-\phi(y))}{|x-y|^{1+2s}}dxdydt=0.
\end{align}
Integrating by parts, we obtain
\begin{multline}\label{ec9}
-\int_0^{\tau}\int_{-1}^1 w\phi_t dxdt +\int_{-1}^1\Big(w(\cdot,\tau)\phi(\cdot,\tau)-w(\cdot,0)\phi(\cdot,0)\Big)dx \\+\frac{C_s}{2}\int_0^{\tau}\int_{-1}^1\int_{-1}^1 \frac{(w(x)-w(y))(\phi(x)-\phi(y))}{|x-y|^{1+2s}}dxdydt =0.
\end{multline}
Because $\phi$ is the solution of \eqref{null-control} and satisfies $\phi(\cdot,0)=0$ a.e. in $(-1,1)$, we have:
\begin{align}\label{ec10}
-\int_{\mathcal{\omega}}\int_0^{\tau}h^{(\tau)}(x,t)w(x,t)dtdx = (w(\cdot,\tau),\phi(\cdot,\tau))_{L^2(-1,1)}=(w(\cdot,\tau),\varphi_k)_{L^2(-1,1)}.
\end{align}

Now, from Lemma \ref{Lemma2} there exists $\theta^{(\tau)}\in H^1(0,T;L^2(-1,1))$, with $\theta^{\tau}(\cdot,t)=0$ for $t\in[\tau,T]$, such that
\begin{align}\label{ec11}
h^{(\tau)}=K^{*}\theta^{(\tau)}.
\end{align}
Therefore, 
\begin{align}\label{ec12}
(w(\cdot,\tau),\varphi_k)_{L^2(-1,1)}&=-(w,h^{(\tau)})_{L^2(0,T;L^2(\omega))}\notag\\
&= -(w,K^{*}\theta^{(\tau)})_{L^2(0,T;L^2(\omega))}\notag\\
&=-\langle Kw,\theta^{(\tau)}\rangle_{H^1(0,T;L^2(\omega))}\notag\\
&=-\langle u,\theta^{(\tau)}\rangle_{H^1(0,T;L^2(\omega))}.
\end{align}
Finally, from \eqref{ec6}, \eqref{ec7}, and \eqref{ec12}, the unknown Fourier coefficient of $f$ is given by:
\begin{multline*}
(f,\varphi_k)_{L^2(-1,1)}=-\frac{\sigma(0)}{\sigma(T)} \langle u,\theta^{T}\rangle_{H^1(0,T;L^2(\omega))}-\frac{1}{\sigma(T)}\int_0^T \sigma'(T-\tau)\langle u,\theta^{T}\rangle_{H^1(0,T;L^2(\omega))}d\tau\\-\frac{\lambda_k}{\sigma(T)}\int_0^T \sigma(T-\tau)\langle u,\theta^{T}\rangle_{H^1(0,T;L^2(\omega))}d\tau,
\end{multline*}
and the proof was complete.
\end{proof}

Using the series representation of the solution for the nonlocal heat equation given in \eqref{series1}, the previous reconstruction formula can be reformulated as follows.

\begin{theorem}\label{th2}
Under the same assumptions as those in Theorem \ref{th1}. If, for each $n\in\N$ we have:
\begin{align}\label{th2.1}
c_n:=\sigma(T)-\lambda_n\int_0^Te^{\lambda_n(s-T)}\sigma(s)d s\neq0,
\end{align}
then the following reconstruction formula holds:
\begin{equation}\label{th2.2}
\langle f,\varphi_n\rangle_{L^2(-1,1)}=
c_n^{-1}\left(-\sigma(0)\langle u,\theta^{(T)}\rangle_{H^1(0,T;L^2(\omega))}
-\int_0^T\sigma'(T-\tau)\langle u,\theta^{(\tau)}\rangle_{H^1(0,T;L^2(\omega))}d \tau\right).
\end{equation}
\end{theorem}

\begin{proof}
From Remark \ref{r1}--(d), the solution of \eqref{HE-EX} is given by
\begin{align}\label{ec13}
u(x,t) = \sum_{k\geq 1}\left( f_{k} \int_0^t e^{-\lambda_k(t-s)}\sigma(s)ds \right)\varphi_k(x).
\end{align}
Then, the last term on the right-hand side of \eqref{ec6} can be replaced by
\begin{align}\label{ec14}
((-\partial_x^2)^s u(\cdot,T),\varphi_k)_{L^2(-1,1)}&=\lambda_k(u(\cdot,T),\varphi_k)_{L^2(-1,1)}\notag\\
&= \lambda_k \Big(\sum_{n\geq 1}\left( f_{n} \int_0^T e^{-\lambda_n(T-s)}\sigma(s)ds \right)\varphi_n(x),\varphi_k\Big)_{L^2(-1,1)}\notag\\
&=\lambda_k f_k \int_0^T e^{-\lambda_k(T-s)}\sigma(s)ds.
\end{align}
Therefore, from \eqref{th1.1} the following identity is obtained
\begin{multline*}
\sigma(T)( f,\varphi_n)_{L^2(-1,1)}-\lambda_n ( f,\varphi_n)_{L^2(-1,1)}\int_0^T e^{-\lambda_k(T-s)}\sigma(s)ds=
-\sigma(0)\langle u,\theta^{(T)}\rangle_{H^1(0,T;L^2(\omega))}\\
-\int_0^T\sigma'(T-\tau)\langle u,\theta^{(\tau)}\rangle_{H^1(0,T;L^2(\omega))}d \tau.
\end{multline*} 
Thus, since $c_n\neq 0$, that is,
\begin{align*}
\sigma(T)-\lambda_n \int_0^T e^{-\lambda_k(T-s)}\sigma(s)ds\neq 0,
\end{align*}
we conclude the proof of the Theorem.
\end{proof}


\section{Optimal control problem: penalized HUM approach}\label{sec-4}

This section presents the formulation of an optimal control problem, which serves as the basis for designing a numerical algorithm to implement the reconstruction formula derived in the previous section. The reconstruction procedure relies on several sub-problems: solving the eigenvalue problem for the fractional Laplacian, computing a family of null controls, and solving the associated Volterra integral equation for each case. In this section, we focus on the family of control problems.

It is well known (see, for instance, \cite{glowinski2008exact, boyer2013}) that null controls can be computed using the penalized Hilbert Uniqueness Methodology (HUM). This approach formulates the control problem as a convex quadratic penalized optimization problem designed to characterize and compute the control of the minimal $L^2$-norm. 
	
We denote by $W(0,\tau)$ the space of all $y\in L^2(0,\tau; H_0^s(-1,1))$ with $y_t\in L^2(0,\tau;H^{-s}((-1,1)))$.
Consider the following optimal control problem
\begin{align}\label{funct}
\min_{h\in L^2(0,\tau;L^2(\omega))} \left(J_{\varepsilon}(h):=\frac{1}{2}\int_0^{\tau}\int_{\omega}|h|^2dxdt+\frac{\varepsilon}{2}\|\phi(\cdot,0)\|_{L^2(\Omega)}^2\right),
\end{align}
subject to
\begin{align}\label{control}
\begin{cases}
-\partial_t \phi + (-\partial_x^2)^{s} \phi = h^{(\tau)}\chi_{\omega\times(0,\tau)} & \mbox{ in }\; (-1,1)\times(0,\tau),\\
\phi=0 &\mbox{ in }\; (\Omc)\times (0,\tau), \\
\phi(\cdot,\tau) = \varphi_0&\mbox{ in }\; (-1,1),
\end{cases}
\end{align}
where $\varphi_0\in L^2(-1,1)$ is fixed.

The solution mapping $(h^{\tau},\varphi_0)\mapsto \phi$ corresponding to the initial--boundary value problem \eqref{control} has the following structure
\begin{align*}
y=Gh+G_0\varphi_0,
\end{align*}
where the continuous linear operators $G:L^2(0,\tau;L^2(-1,1))\to W(0,\tau)$ and $G_0:L^2(-1,1)\to W(0,\tau)$ are defined, respectively, by
\begin{align*}
G: h\mapsto y\quad \text{for }\varphi_0=0,\\
G_0: \varphi_0\mapsto y\quad \text{for }h=0.
\end{align*}

Because only the final state $\phi(\cdot,0)$ appears in the cost functional, we proceed as in \cite[Chapter III]{troltzsch2010optimal} and define the observation operator $E_0: \phi\mapsto \phi(\cdot,0)$. For $\varphi_0=0$, this mapping is linear and continuous from $W(0,\tau)$ in $L^2(-1,1)$. Thus, we have
\begin{align}
\phi(\cdot,0)=E_0(Gh)=:Sh.
\end{align}
In other words, $S$ represents the part of the state that appears in the functional. Therefore, we can rewrite the optimal control problem \eqref{funct}-\eqref{control} as follows: Let $\varphi_0\in L^2(-1,1)$ be fixed. We then consider the following optimal control problem:
\begin{align}\label{funct-2}
\min_{h\in L^2(0,\tau;L^2(\omega))} \left(J_{\varepsilon}(h):=\frac{1}{2}\int_0^{\tau}\int_{\omega}|h|^2dxdt+\frac{\varepsilon}{2}\|Sh-z\|_{L^2(\Omega)}^2\right),
\end{align}
where $z=-G_0\varphi_0$

Hence, we can deduce from \cite[Theorem 2.14]{troltzsch2010optimal} the following existence result.
\begin{theorem}\label{th3}
Let $\varepsilon>0$. Then there exists a unique solution $\overline{h}_{\varepsilon}^{(\tau)}$ to \eqref{funct-2}, and hence the optimal nonstationary problem \eqref{funct}-\eqref{control}. 
\end{theorem}

In addition, using Theorem 3.19 and 3.20 in \cite{troltzsch2010optimal}, we obtain the following first-order necessary condition.
\begin{theorem}\label{th4}
Let $\overline{h}_{\varepsilon}^{(\tau)}$ be the minimizer of \eqref{funct} over $L^2(0,\tau;L^2(\omega))$ with the associated state $\overline{\phi}_{\varepsilon}$. Then, for every $\varepsilon>0$, the pair $(\overline{h}_{\varepsilon}^{(\tau)},\overline{\phi}_{\varepsilon})$ satisfies the optimality system
\begin{align}\label{control-3}
\begin{cases}
-\partial_t \overline{\phi}_{\varepsilon} + (-\partial_x^2)^{s} \overline{\phi}_{\varepsilon} = \overline{h}_{\varepsilon}^{(\tau)}\chi_{\omega\times(0,\tau)} & \mbox{ in }\; (-1,1)\times(0,\tau),\\
\overline{\phi}_{\varepsilon}=0 &\mbox{ in }\; (\Omc)\times (0,\tau), \\
\overline{\phi}_{\varepsilon}(\cdot,\tau) = \varphi_0&\mbox{ in }\; (-1,1),
\end{cases}
\end{align}
and $\overline{p}_{\varepsilon}\in C([0,\tau];L^2(-1,1))$ solves the adjoint problem
\begin{align}\label{dual-1}
\begin{cases}
\partial_t \overline{p}_{\varepsilon} + (-\partial_x^2)^{s} \overline{p}_{\varepsilon} = 0 & \mbox{ in }\; (-1,1)\times(0,\tau),\\
\overline{p}_{\varepsilon}=0 &\mbox{ in }\; (\Omc)\times (0,\tau), \\
\overline{p}_{\varepsilon}(\cdot,0) = \varepsilon\overline{\phi}_{\varepsilon}(\cdot,0)&\mbox{ in }\; (-1,1),
\end{cases}
\end{align}
with the projection formula
\begin{align}\label{fop}
\overline{h}_{\varepsilon}^{(\tau)}+\overline{p}_{\varepsilon}=0,\quad \text{in }\omega\times(0,\tau).
\end{align}
\end{theorem}

Now, under the existence and first-order optimality condition, the following theorem shows that the optimal solutions $(\overline{h}_\varepsilon, \overline{\phi}_\varepsilon, \overline{p}_\varepsilon)$ of the optimal control problem \eqref{funct} indexed in $\varepsilon$ will help us to design the numerical algorithm to numerically compute the reconstruction formula. 

\begin{theorem}\label{th5}
Let $\overline{h}_{\varepsilon}^{(\tau)}$ be the minimizer of \eqref{funct} over $L^2(0,\tau;L^2(\omega))$ with the associated state $\overline{\phi}_{\varepsilon}$. Let $\overline{p}_{\varepsilon}\in C([0,\tau];L^2(-1,1))$ be the solution to the adjoint problem \eqref{dual-1}. Then, when $\varepsilon$ tends to infinity, the following convergence holds:
\begin{multline}\label{reconsformula}
-\frac{\sigma(0)}{\sigma(T)}(u,\theta_{\varepsilon}^{(\tau)})_{L^2(\omega)}-\frac{1}{\sigma(T)}\int_0^T \sigma'(T-\tau)(u,\theta_{\varepsilon}^{(\tau)})_{L^2(\omega)}d\tau\\-\frac{1}{\sigma(T)}((-\partial_x^2)^s u(\cdot,T),\varphi_0)_{L^2(-1,1)} \to (f,\varphi_0)_{L^2(-1,1)}.
\end{multline}
\end{theorem}

\begin{proof}
From the optimality system given in Theorem \ref{th4}, we multiply \eqref{control-3} by $p_{\varepsilon}$, integrate over $(-1,1)\times(0,\tau)$, and apply the integration by parts formula \eqref{Int-Part} to obtain:
\begin{align}\label{th5-1}
\int_0^{\tau}\int_{\omega} |h_{\varepsilon}^{(\tau)}|^2dxdt + \varepsilon\|\phi_{\varepsilon}(\cdot,0)\|_{L^2(-1,1)}^2 = \int_{-1}^1 p_{\varepsilon}(x,\tau)\varphi_0(x)dx.
\end{align}
We observe that the right-hand side of \eqref{th5-1} can be bounded as follows:
\begin{align}\label{th5-2}
\int_{-1}^1 p_{\varepsilon}(x,\tau)\varphi_0(x)dx\leq \frac{\delta}{2}\|p_{\varepsilon}(\cdot,\tau)\|_{L^2(-1,1)}^2+ \frac{1}{\delta}\|\varphi_0\|_{L^2(-1,1)}^2,
\end{align}
for every $\delta>0$. From the observability inequality \eqref{obs-ine} to system \eqref{dual-1}, there exists 
a constant $C=C(\tau)>0$, uniformly bounded away from $\tau=0$, such that
\begin{align}\label{th5-3}
\|p_{\varepsilon}(\cdot,\tau)\|_{L^2(-1,1)}^2 \leq C\int_0^{\tau}\int_{\omega}|p_{\varepsilon}|^2dxdt.
\end{align}
Because $p_{\varepsilon}=-h_{\varepsilon}^{(\tau)}$ a.e. in $\omega\times(0,\tau)$, substituting \eqref{th5-3} into \eqref{th5-2}, and then in \eqref{th5-1}, we obtain
\begin{align*}
\int_0^{\tau}\int_{\omega} |h_{\varepsilon}^{(\tau)}|^2dxdt +\varepsilon\|\phi_{\varepsilon}\|_{L^2(-1,1)}^2\leq \frac{\delta C(\tau)}{2}\int_0^{\tau}\int_{\omega} |h_{\varepsilon}^{(\tau)}|^2dxdt+\frac{1}{2\delta}\|\varphi_0\|_{L^2(-1,1)}^2.
\end{align*}
Taking $\delta=\frac{1}{C(\tau)}$, we get
\begin{align}\label{th5-4}
\frac{1}{2\varepsilon}\int_0^{\tau}\int_{\omega} |h_{\varepsilon}^{(\tau)}|^2dxdt +\|\phi_{\varepsilon}\|_{L^2(-1,1)}^2\leq \frac{ C(\tau)}{2\varepsilon}\|\varphi_0\|_{L^2(-1,1)}^2.
\end{align}

However, because the sequence  $\{h_{\varepsilon}^{(\tau)}\}_{0\leq \tau\leq T}$ is uniformly bounded in $L^2(0,\tau;L^2(\omega))$, we can conclude from the continuity of the solution $\phi_{\varepsilon}$ with respect to the data that $\phi_{\varepsilon}$ is uniformly bounded in $C([0,\tau];H_0^s(-1,1))$. Therefore, by extracting the subsequence of $\{h_{\varepsilon}^{(\tau)}\}$ and $\{\phi_{\varepsilon}\}$, denoted in the same way,  the following holds when $\varepsilon$ approaches infinity :
\begin{align*}
h_{\varepsilon}^{(\tau)} \rightharpoonup h^{(\tau)} \quad \text{weakly in } L^2(0,\tau;L^2(\omega)),\\
\phi_{\varepsilon} \rightharpoonup \phi \quad \text{weakly in } L^2(0,\tau; H_0^s(-1,1)),\\
\partial_t \phi_{\varepsilon} \rightharpoonup \partial_t \phi \quad \text{weakly in } L^2(0,\tau; L^2(-1,1)).
\end{align*}
It is well-known that embedding  $H_0^s(-1,1)\hookrightarrow L^2(-1,1)$ is compact. Using this fact, from \eqref{th5-4}, we obtain 
\begin{align*}
\phi_{\varepsilon}(\cdot,0)\to \phi(\cdot,0) \quad\text{ in }L^2(-1,1).
\end{align*}
Therefore, again from \eqref{th4} and the previous strong convergence, we deduce
\begin{align*}
\|\phi_{\varepsilon}(\cdot,0)\|_{L^2(-1,1)}\to 0\quad \text{ as }\varepsilon\to\infty.
\end{align*}
That is, $\phi(\cdot,0)=0$. Finally, taking $\eta=h_{\varepsilon}^{(\tau)}$ in Lemma \ref{Lemma1}, we can extract a subsequence of $\{\theta_{\varepsilon}^{(\tau)}\}$ such that
\begin{align*}
\theta_{\varepsilon}^{(\tau)} \rightharpoonup \theta^{(\tau)}\quad \text{weakly in }H^1(0,\tau;L^2(-1,1)),
\end{align*}
which implies the desired result \eqref{reconsformula}. The proof is finished.

\end{proof}

\section{Numerical Reconstruction Scheme}\label{sec-5}

As shown, recovering the Fourier coefficients of the spatial component of the source term involves several computational challenges. In this section, we present a detailed numerical scheme for addressing these issues.

\subsection{Numerical solution of the nonlocal heat equation}\label{sec-5.1}

Let us denote by $\Omega=(-1,1)$. To solve the nonlocal heat equation
\begin{align}\label{eq5.1}
\begin{cases}
u_t + (-\partial_x^2)^s u = F(x, t) & \text{in } \Omega \times (0, T), \\
u = 0 & \text{in } (\mathbb{R} \setminus \Omega) \times (0, T), \\
u(\cdot, 0) = u_0 & \text{in } \Omega.
\end{cases}
\end{align}
We used a finite element method for spatial discretization and an implicit Euler scheme for the time component based on the approach in \cite{biccaricontrollability}. The key steps are summarized below.

We begin by approximating the solution to the nonlocal Poisson equation
\begin{align}\label{eq5.2}
\begin{cases}
(-\partial_x^2)^s u = f & \text{in } \Omega, \\
u = 0 & \text{in } \mathbb{R} \setminus \Omega.
\end{cases}
\end{align}
The weak (variational) formulation seeks  \( u \in H^s_0(\Omega) \) such that
\begin{align*}
a(u, v) = \int_\Omega f(x) v(x) \, dx, \quad \forall v \in H^s_0(\overline\Omega),
\end{align*}
where \( a : H^s_0(\overline\Omega) \times H^s_0(\overline\Omega) \to \mathbb{R} \) is the bilinear form defined by
\begin{align*}
a(u, v) = C_s \int_{\mathbb{R}} \int_{\mathbb{R}} \frac{(u(x) - u(y))(v(x) - v(y))}{|x - y|^{1 + 2s}} \, dx \, dy.
\end{align*}
Because the bilinear form \( a(\cdot, \cdot) \) is continuous and coercive, the Lax-Milgram theorem guarantees a unique solution to this problem. More specifically, if \( f \in H^{-s}(\overline\Omega) \), then there exists a unique weak solution \( u \in H^s_0(\overline\Omega) \) (see \cite{BWZ1}).

Given this result, the finite element method (FEM) seeks approximate solutions by restricting the problem to a discrete subspace. Let \( N \in \mathbb{N} \) and consider a uniform mesh \( \mathfrak M := \{x_i\}_{i=0}^{N+1} \) with step size \( h = \frac{2}{N+1} \) and $x_i = -1 + hi$. We define the discrete space of the continuous and piecewise linear functions \( V_h \) as follows:
\[
V_h := \left\{ v \in H^s_0(\overline\Omega) : v|_{K_i} \in P_1, \; i = 0, 1, \ldots, N \right\},
\]
where \( K_i = [x_i, x_{i+1}] \). We then approximate the problem by solving for \( u_h \in V_h \) such that
\[
a(u_h, v_h) = \int_\Omega f(x) v_h(x) \, dx \quad \forall v_h \in V_h.
\]

Let \( \{\phi_i\}_{i=1}^N \) be the basis for \( V_h \), where each basis function is defined as
\[
\phi_i(x) = 
\begin{cases}
1 - \frac{|x - x_i|}{h} & \text{if } x \in (x_{i-1}, x_{i+1}), \\
0 & \text{otherwise}.
\end{cases}
\]
Then \( u_h \) can be expressed as a linear combination of these basis functions:
\[
u_h(x) = \sum_{j=1}^N u_j \phi_j(x),
\]
where \( u_j = u_h(x_j) \). By testing this expression with \( v_h = \phi_i \) for \( i = 1, \ldots, N \), we obtain the linear system
\[
A_h U = F,
\]
where \( A_h \) is a symmetric stiffness matrix with entries \( a_{i,j} = a(\phi_i, \phi_j) \), and \( F \) is a vector with components \( F_i = \int_\Omega f(x) \phi_i(x) \, dx \).

The components of the stiffness matrix \( A_h \in \mathbb{R}^{N \times N} \) depend only on the difference \( k = i - j \) and spacing \( h > 0 \) between the points (except for the case \( s = \frac{1}{2} \)), with specific values calculated in \cite{biccaricontrollability}.

Additionally, if we express \( f(x) \) as a linear combination of the finite element basis functions \( f(x) = \sum_{j=1}^N f_j \phi_j(x) \), we obtain
\[
\int_\Omega f(x) \phi_i(x) \, dx = \frac{h}{6} \left( f(x_{i-1}) + 4 f(x_i) + f(x_{i+1}) \right), \quad i = 1, \ldots, N,
\]
resulting in the mass matrix \( M_h \), a tridiagonal matrix given by
\[
M_h = h
\begin{pmatrix}
2/3 & 1/6 & 0 & \cdots & 0 \\
1/6 & 2/3 & 1/6 & \cdots & 0 \\
0 & 1/6 & 2/3 & \cdots & 0 \\
\vdots & \vdots & \vdots & \ddots & \vdots \\
0 & 0 & 0 & 1/6 & 2/3
\end{pmatrix}.
\]
This matrix satisfies the relation \( F = M_h f \), where \( f = (f_1, f_2, \dots, f_N)^\top \) is the vector of the coefficients of \( f(x) \) on a finite-element basis.

Using this approach, we can apply these matrices similar to the variational equation associated with the nonlocal heat equation \eqref{eq5.1}, yielding the following initial value problem:
\[
\begin{cases}
M_h U_t(t) + A_h U(t) = M_h F(t), & t \in (0, T), \\
U(0) = U_0,
\end{cases}
\]
where \( U(t) = (u(x_1, t), \dots, u(x_N, t))^\top \) and \( F(t) = (F(x_1, t), \dots, F(x_N, t))^\top \).

Finally, for uniform time discretization \( 0 = t_0 < t_1 < \dots < t_j = k j < \dots < t_M = T \), where \( k = T / M \), and we define \( U^j = U(t_j) \). By applying an implicit Euler discretization, we obtain
\[
\begin{cases}
M_h \displaystyle\frac{U^{j+1} - U^j}{k} + A_h U^{j+1} = F^{j+1}, & j = 0, \dots, M - 1, \\
U^0 = U_0.
\end{cases}
\]
This leads to the iterative solution of the linear system
\[
(M_h + k A_h) U^{j+1} = M_h U^j + k F^{j+1}, \quad U^0 = U_0.
\]
The above procedure allows for an approximation of the solution \( u \) to equation \eqref{eq5.1}. Additionally, by employing a similar process, we can discretize the differential equation associated with the control problem, using backward time discretization.

\subsection{Approximation of eigenfunctions}\label{sec-5.2}

In Section \ref{sec-5.1}, we discussed that to solve Poisson equation \eqref{eq5.2} numerically, it is necessary to solve the linear system
\[
A_h U = M_h F.
\]
Similarly, to solve the eigenvalue problem
\[
\begin{cases}
(-\partial_x^2)^s \varphi = \lambda \varphi & \text{in } \Omega, \\
\varphi = 0 & \text{in } \mathbb{R} \setminus \Omega.
\end{cases}
\]
We computed the eigenvalues of matrix \( D = M_h^{-1} A_h \), with the eigenvectors serving as approximations for the eigenfunctions \( \varphi_n \) of \((-\partial_x^2)_D^s\).

This approach is used to approximate the eigenvalues and eigenfunctions of the fractional Laplacian. However, it is essential to normalize the eigenvectors properly as functions of \( L^2(\Omega) \). This can be achieve by calculating their norms using integration or leveraging the structure of the approximation space, \( V_h \). Specifically, if we write an eigenvector as \( \varphi_n(x) = \sum_{j=1}^N \varphi_{n,j} \phi_j(x) \), then the norm is given by
\[
\|\varphi_n\|_{L^2(\Omega)}^2 = \int_{-1}^1 |\varphi_n(x)|^2 \, dx,
\]
which can be expressed as
\begin{align*}
\|\varphi_n\|_{L^2(\Omega)}^2 = h \left( \sum_{j=1}^N \frac{2}{3} |\varphi_{n,j}|^2 + \sum_{j=1}^{N-1} \frac{2}{6} \varphi_{n,j} \varphi_{n,j+1} \right)
=\begin{pmatrix}\varphi_{n,1}&\cdots&\varphi_{n,N}\end{pmatrix}M_h\begin{pmatrix}\varphi_{n,1}\\\vdots\\\varphi_{n,N}\end{pmatrix}.
\end{align*}
Thus, we obtain a normalized \( \varphi_n \) by using only the mass matrix \( M_h \).

\subsection{Approximation with partial sums}
\label{PSapprox}
As previously established in equation \eqref{series1}, the solution to the nonlocal heat equation can be expressed in terms of its Fourier coefficients. Similarly, this approach can be applied to the following system.
$$
\begin{array}{lc}
\multicolumn{1}{c}{\textbf{Differential equation}}&\textbf{Spectral solution}
\\\vspace{-5px}\\\fontsize{9.8}{13}\selectfont
\left\{\begin{array}{ll}
u_t + \flap u = 0&\text{in }\Omega\times(0,T),\\
u = 0							&\text{in }(\R\setminus\Omega)\times(0,T),\\
u(\cdot,0) = u_0				&\text{in }\Omega.
\end{array}\right.
&\displaystyle u(x,t)=
\sum_{n\geq1} u_{0,n}e^{-\lambda_nt}\varphi_n(x).
\\\vspace{-5px}\\\fontsize{9.8}{13}\selectfont
\left\{\begin{array}{ll}
-\phi_t + \flap \phi = f(x,t)&\text{in }\Omega\times(0,\tau),\\
\phi = 0							&\text{in }(\R\setminus\Omega)\times(0,\tau),\\
\phi(\cdot,\tau) = 0				&\text{in }\Omega.
\end{array}\right.
&\displaystyle \phi(x,t)=
\sum_{n\geq1}\left(\int_t^\tau e^{\lambda_n(t-s)}\left(f(\cdot,s)\right)_n\,ds\right)\varphi_n(x).
\\\vspace{-5px}\\\fontsize{9.8}{13}\selectfont
\left\{\begin{array}{ll}
-\phi_t + \flap \phi = 0&\text{in }\Omega\times(0,\tau),\\
\phi = 0							&\text{in }(\R\setminus\Omega)\times(0,\tau),\\
\phi(\cdot,\tau) = \varphi_n				&\text{in }\Omega.
\end{array}\right.
&\phi(x,t)=
e^{\lambda_n(t-\tau)}\varphi_n(x).\\~\vspace{-5px}
\end{array}$$

By employing a set of eigenfunctions up to a given number $n^*$, the spectral solution can be approximated using a partial sum. In the followings, we truncate the series to $n^* = N/5$ eigenfunctions, where $N+2$ represents the size of the mesh $\mathfrak M$.

\subsection{Numerical solution of the optimal control problem}

Let \( \varepsilon > 0 \), \( \tau > 0 \), and \( \varphi_0 \in L^2(\Omega) \) be fixed. As discussed in the previous section, the optimal control problem aims to minimize functional
\[
J(\phi, h) = \frac{1}{2} \int_0^{\tau} \int_\omega |h|^2 \, dx \, dt + \frac{\varepsilon}{2} \|\phi(\cdot, 0)\|_{L^2(\Omega)}^2,
\]
subject to the state equation
\[
\begin{cases}
-\phi_t + (-\partial_x^2)^s \phi = h \chi_{\omega \times (0, \tau)} & \text{in } \Omega \times (0, \tau), \\
\phi = 0 & \text{in } (\mathbb{R} \setminus \Omega) \times (0, \tau), \\
\phi(\cdot, \tau) = \varphi_0 & \text{in } \Omega.
\end{cases}
\]

According to the optimality conditions, the minimum satisfies systems \eqref{control-3} and \eqref{dual-1} and projection formula \eqref{fop}. Following \cite{garcia2013heat}, this relationship is used to define a linear system whose solution provides an approximation for \( \varepsilon\phi(\cdot, 0) \). To achieve this, we define the operator
\[
L : L^2(\Omega) \to L^2(\omega \times (0, \tau)), \quad e \mapsto -w \chi_{\omega \times (0, \tau)},
\]
where \( w \) is the solution to the adjoint problem with initial state \( w(\cdot, 0) = e \), and
\[
L' : L^2(\omega \times (0, \tau)) \to L^2(\Omega), \quad h \mapsto -\psi(\cdot, 0),
\]
where \( \psi \) is the solution to the state problem with source \( h \chi_{\omega \times (0, \tau)} \) and initial state \( \psi(\cdot, \tau) = 0 \).

We also define the operator \( \Lambda : L^2(\Omega) \to L^2(\Omega) \) as
\[
\Lambda e = L'Le.
\]

Using these operators, if \( \Phi_0 \) is the solution to the state problem with zero source and initial state \( \Phi_0(\cdot, \tau) = \varphi_0 \), then \( \phi \) can be expressed as
\[
\phi = \psi + \Phi_0,
\]
where \( \psi \) is the solution to the state problem with source \( h \chi_{\omega \times (0, \tau)} \) and initial state $\psi(\cdot,\tau) = 0$.

This yields the following relationship:
\[
\varepsilon^{-1} \varepsilon \phi(\cdot, 0) = -L' h + \Phi_0(\cdot, 0).
\]
Since \( h = -p \) on \( \omega \times (0, \tau) \) and by definition \( p = -L(\varepsilon \phi(\cdot, 0)) \), finding \( w_0 = \varepsilon \phi(\cdot, 0) \) is equivalent to solving
\[
\varepsilon^{-1} w_0 = -L' L w_0 + \Phi_0(\cdot, 0),
\]
or
\[
\left( \varepsilon^{-1} I + \Lambda \right) w_0 = \Phi_0(\cdot, 0).
\]
Finally, control \( h \) can be obtained by evaluating \( h = L w_0 \) in \( \omega \times (0, \tau) \).

The next step is to reduce the previous equation to the discrete space \( V_h \) defined earlier. That is, we seek \( w_0 = \sum_{j=1}^N w_j \phi_j \) such that
\begin{align}\label{eq5.3}
\sum_{j=1}^N w_j \left[ \left( \varepsilon^{-1} I + \Lambda \right) \phi_j \right](x_i) = \Phi_0(x_i, 0), \quad i = 1, \dots, N.
\end{align}

Thus, if for each \( j = 1, \dots, N \) we calculate \( p_j \) as the solution of the adjoint problem with initial state \( p_j(\cdot, 0) = \phi_j \) and \( \Phi_j \) as the solution of the direct problem with source \( p_j \chi_{\omega \times (0, \tau)} \) and zero final condition, it follows immediately that \( \Lambda \phi_j = \Phi_j(\cdot, 0) \) and \eqref{eq5.3} becomes a linear system:
\begin{align}\label{eq5.4}
\begin{pmatrix}
\varepsilon^{-1} I_N + 
\begin{pmatrix}
\Phi_1(x_i, 0) & \Phi_2(x_i, 0) & \cdots & \Phi_N(x_i, 0) \\
\vdots & \vdots & & \vdots \\
\Phi_1(x_N, 0) & \Phi_2(x_N, 0) & \cdots & \Phi_N(x_N, 0)
\end{pmatrix}
\end{pmatrix}
\begin{pmatrix}
w_1 \\
\vdots \\
w_N
\end{pmatrix}
= 
\begin{pmatrix}
\Phi_0(x_i, 0) \\
\vdots \\
\Phi_0(x_N, 0)
\end{pmatrix}.
\end{align}

To calculate \( \Phi_0 \), \( \Phi_j \), and \( p_j \) for \( j = 1, \dots, N \), and consequently approximate \( h \), it is possible to use the Finite Element method discussed in Sections \ref{sec-5.1} and \ref{sec-5.2}. However, carrying this out for each \( \tau_\ell > 0 \), where
\[
0 = \tau_0 < \tau_1 < \cdots < \tau_\ell = \kappa \ell < \cdots < \tau_{f_M} = T,
\]
is a uniform partition containing \( \{ t_j \}_{j=0}^M \), with \( \kappa = T / f_M \), which requires considerable computation time. Therefore, it is more convenient to express it in terms of the Fourier coefficients.

Therefore, we used expressions based on the basis presented in Section \ref{PSapprox}. It follows that, for each \( j = 1, \dots, N \)
\[
p_j(x, t) = \sum_{m\geq1} \rho_m(t) \varphi_m(x), \quad \rho_m(t) = (\phi_j)_m e^{-\lambda_m t},
\]
\[
\Phi_j(x, t) = \sum_{n\geq1} \zeta_n(t) \varphi_n(x), \quad \zeta_n(t) = \int_\tau^t e^{-\lambda_n(t-s)} ( p_j(\cdot, s), \varphi_n )_{L^2(\omega)} \, ds,
\]
\[
\Phi_0(x, 0) = e^{-\lambda_n \tau} \varphi_n(x), \quad \text{if } \varphi_0 = \varphi_n,
\]
where \( (\phi_j)_m = ( \phi_j, \varphi_m )_{L^2(\Omega)} \). Note that
\[
( p_j(\cdot, s), \varphi_n )_{L^2(\omega)} = \sum_{m=1}^{\infty} \rho_m(s) \int_\omega \varphi_n(x) \varphi_m(x) \, dx = \sum_{m\geq1} (\phi_j)_m\, e^{-\lambda_m s} ( \varphi_n, \varphi_m )_{L^2(\omega)}.
\]

Then, using the approximation of \( \varphi_n \) from the previous section, we can calculate the integral \( ( \varphi_n, \varphi_m )_{L^2(\omega)} \) using the trapezoidal rule. Applying this rule to \( (\phi_j)_m \), we obtain
\[
(\phi_j)_m = \int_{x_{j-1}}^{x_{j+1}} \phi_j(x) \varphi_m(x) \, dx \approx h \varphi_m(x_j).
\]
However, given the nature of \( \varphi_j \) and the integration rule in \ref{apA}, we obtain a more precise value:
\[
(\varphi_j)_m \approx \frac{1}{12} \left( \varphi_m(x_{j-1}) + 10 \varphi_m(x_j) + \varphi_m(x_{j+1}) \right).
\]

Therefore,
\[
\zeta_n(0) = \sum_{m=1}^{\infty} (\phi_j)_m \langle \varphi_n, \varphi_m \rangle_{L^2(\omega)} \int_\tau^0 e^{-(\lambda_n + \lambda_m) s} \, ds = \sum_{m\geq1} (\phi_j)_m \langle \varphi_n, \varphi_m \rangle_{L^2(\omega)} \frac{1 - e^{-(\lambda_n + \lambda_m) \tau}}{\lambda_n + \lambda_m}.
\]

Then,
\[
\Phi_j(x, 0) \approx \sum_{n=1}^{n^*} \left( \sum_{m=1}^{n^*} (\phi_j)_m \langle \varphi_n, \varphi_m \rangle_{L^2(\omega)} \frac{1 - e^{-(\lambda_n + \lambda_m) \tau}}{\lambda_n + \lambda_m} \right) \varphi_n(x).
\]

Incorporating this expression into system \eqref{eq5.4} and solving it, we only need to compute \( h = L w_0 \) in \( \omega \times (0, \tau) \):
\[
h(x, t) \approx \sum_{n=1}^{n^*} (w_0)_n\,e^{-\lambda_n t} \varphi_n(x) \quad \text{in } \omega \times (0, \tau).
\]

\subsection{Integral equation}

Following the method in \cite{garcia2013heat}, we consider the Volterra integral equation
\[
\begin{cases}
h(x, t) = \sigma(0) \theta_t(x, t) + \displaystyle\int_t^{\tau} \sigma(s - t) \theta(x, s) + \sigma'(s - t) \theta_t(x, s) \, ds, & 0 < t < \tau, \\
\theta(x, \tau) = 0,
\end{cases}
\]
for each \( \tau \in (0, T] \) and \( x \in \omega \), as indicated by Lemma \ref{Lemma1}.

The solution to this equation is denoted as $\theta^{(\tau)}$. Note that if \( \sigma(0) \neq 0 \), according to the Fundamental Theorem of Calculus,
\[
0 = \theta^{(\tau)}(x, t) + \int_t^{\tau} \theta^{(\tau)}_t(x, s) \, ds,
\]
which, coupled with the above integral equation, forms the system
\[
\begin{pmatrix}
0 \\
h^{(\tau)}(x, t)
\end{pmatrix}
= L_0 \begin{pmatrix}
\theta^{(\tau)}(x, t) \\
\theta^{(\tau)}_t(x, t)
\end{pmatrix}
+ \int_t^{\tau} M(t, s) \begin{pmatrix}
\theta^{(\tau)}(x, s) \\
\theta^{(\tau)}_t(x, s)
\end{pmatrix} \, ds,
\]
where \( L_0 \) and \( M(t, s) \) are defined as
\[
L_0 = \begin{pmatrix} 1 & 0 \\ 0 & \sigma(0) \end{pmatrix}, \quad M(t, s) = \begin{pmatrix} 0 & 1 \\ \sigma(s - t) & \sigma'(s - t) \end{pmatrix}.
\]
In vector notation, this can be expressed as
\[
\tilde{h}^{(\tau)}(x, t) = L_0 \tilde{\theta}^{(\tau)}(x, t) + \int_t^{\tau} M(t, s) \tilde{\theta}^{(\tau)}(x, s) \, ds, \quad t \in [0, \tau].
\]

Thus, for \( \tau = \tau_{\ell} \), \( \ell = 1, 2, \dots, M \), the solution \( \tilde{h}^{(\tau)}(x, \tau_{\ell}) \) satisfies by trapezoidal rule:
\begin{equation}
\label{volterra52}
\tilde{h}^{(\tau_{\ell})}(x, \tau_j) \approx L_0 \tilde{\theta}^{(\tau_{\ell})}(x, \tau_k) + \frac{\kappa}{2} \begin{bmatrix} M_{j,j} & 2M_{j,j+1} & \dots & 2M_{j,\ell-1} & M_{j,\ell} \end{bmatrix} \begin{bmatrix} \tilde{\theta}^{(\tau_{\ell})}(x, \tau_j) \\ \vdots \\ \tilde{\theta}^{(\tau_{\ell})}(x, \tau_{\ell}) \end{bmatrix},
\end{equation}
for $j=0, 1, \dots, \ell-1$, where \( M_{j,\ell} = M(\tau_j, \tau_{\ell}) \).

Then, with the initial condition $\widetilde\theta^{(\tau_\ell)}(x,\tau_\ell)=\left(\begin{smallmatrix}0\\\sigma(0)^{-1}\end{smallmatrix}\right)$, we write:
$$\mathbb M_\ell =
\begin{bmatrix}
M_{0,0}&2M_{0,1}&\cdots&2M_{0,\ell-1}&M_{0,\ell}\\
0&M_{1,1}&&2M_{1,\ell-1}&M_{1,\ell}\\
0&0&\ddots&\vdots&\vdots\\
\vdots&\vdots&&M_{\ell-1,\ell-1}&M_{\ell-1,\ell}\\
0&0&\cdots&0&0
\end{bmatrix}
\quad\text{ y }\quad
\mathbb L_\ell=
\begin{bmatrix}
L_0\\
&\ddots\\
&&L_0
\end{bmatrix}
\left.\begin{matrix}\\\\\end{matrix}\right\}~\ell+1\text{ times},
$$
so that \eqref{volterra52} becomes
\begin{equation}
\label{volterra53}
\left(\mathbb L_\ell+\frac{\kappa}2\mathbb M_\ell\right)
\begin{bmatrix}
\widetilde \theta^{(\tau_\ell)}(x,\tau_0)\\
\vdots\\
\widetilde \theta^{(\tau_\ell)}(x,\tau_\ell)
\end{bmatrix}
=
\begin{bmatrix}
\widetilde h^{(\tau_\ell)}(x,\tau_0)\\
\vdots\\
\widetilde h^{(\tau_\ell)}(x,\tau_\ell)
\end{bmatrix}.
\end{equation}
Thus, to approximate $\theta^{(\tau_\ell)}$ and $\theta_t^{(\tau_\ell)}$, it is necessary to solve 
\eqref{volterra53}
for each $x=x_i\in\omega$.

It is noteworthy that when $\ell$ corresponds to the initial terms, the discretized equation contains limited information. For example, when $\ell = 2$, the only information available for $\theta^{(\tau_\ell)}$ is at $t=0$ and $t=\tau_1$. Moreover, as seen in the subsequent calculations, the control exhibits a significantly higher amplitude as $\tau\to0$, making this lack of information non-trivial. To address this, the approach has been slightly modified to solve, for each $\tau=\tau_\ell$, $\ell=1,2,\dots,f_M$, but evaluated at $t=t_0,t_1,\dots,\tau_\ell$, since the second partition is contained within the first.

The numerical schemes presented enable the reconstruction of the Fourier coefficients of the source term in the fractional heat equation as well as the solution of the associated optimal control problem. The finite element method provides accurate approximations of both the state and adjoint systems, whereas the integral formulation facilitates efficient computation of the control variable. These aspects are the focus of the next section, in which we perform numerical experiments to validate our theoretical findings.


\section{Numerical Experiments}\label{sec-6}

In this section, we apply the techniques developed in previous sections to validate the proposed reconstruction methodology. As previously mentioned, the reconstruction formulas derived in Theorems \ref{th1} and \ref{th2} are valid only for \( s > 1/2 \). However, we also provide observations for the operator \((-\partial_x^2)^s\) when \( s \leq 1/2 \). 

We highlight the minimum data required to implement each methodology presented in the previous section and their dependencies in the case of variations in the calculations.

\begingroup
\renewcommand{\arraystretch}{1.5}
\begin{table}[h!]
\centering
\begin{tabular}{|c|c|c|}
\hline
Variable & Data Quantity & Dependency \\
\hline
Mass \( M_h \) and Stiffness \( A_h \) Matrices & \( 2 \times N \times N \) & \( s, \Omega \) \\
Eigenfunctions \( \varphi_n \) & \( N \times N \) & \( M_h, A_h \) \\
Matrices \( \Lambda_N \) & \( f_M \times (N \times N) \) & \( \omega, T, \tau_\ell, \varphi_n, \varepsilon_j \) \\
Control Solution \( h(\tau_\ell) \) & \( f_M \times e_n \times (N \times M) \) & \( \varphi_n, \Lambda_N, \varepsilon \) \\
Volterra Equation Solution \( \theta(\tau_\ell) \) & \( 2 \times f_M \times e_n \times (N \times M) \) & \( h(\tau_\ell), \sigma \) \\
Observed Solution \( u \) and \( u_t \) & \( 2 \times N \times M \) & \( \omega, T \) \\
Approximation of Fourier Coefficients \( f_n \) & \( e_n \) & \( u, u_t, \theta(\tau_\ell) \) \\
\hline
\end{tabular}

\caption{Size of each variable and its dependency.}
\end{table}
\endgroup

\subsection{Spectrum of the Fractional Laplacian}

Using the finite element method with \( N = 500 \) nodes, we computed the eigenvalues of the discrete fractional Laplacian, as shown in Figure \ref{fig01}. As expected, they show a monotonic increasing trend and closely follow the theoretical asymptotic behavior described in \eqref{lam}
\begin{equation}
\tilde{\lambda}_n = \left( \frac{n \pi}{2} - \frac{(1 - s) \pi}{4} \right)^{2s}, \quad n = 1, 2, \ldots,
\label{aproxeigenvalues}
\end{equation}
as long as \( n \) is sufficiently large.
\begin{figure}[!h]
\begin{center}
\includegraphics[scale=0.6]{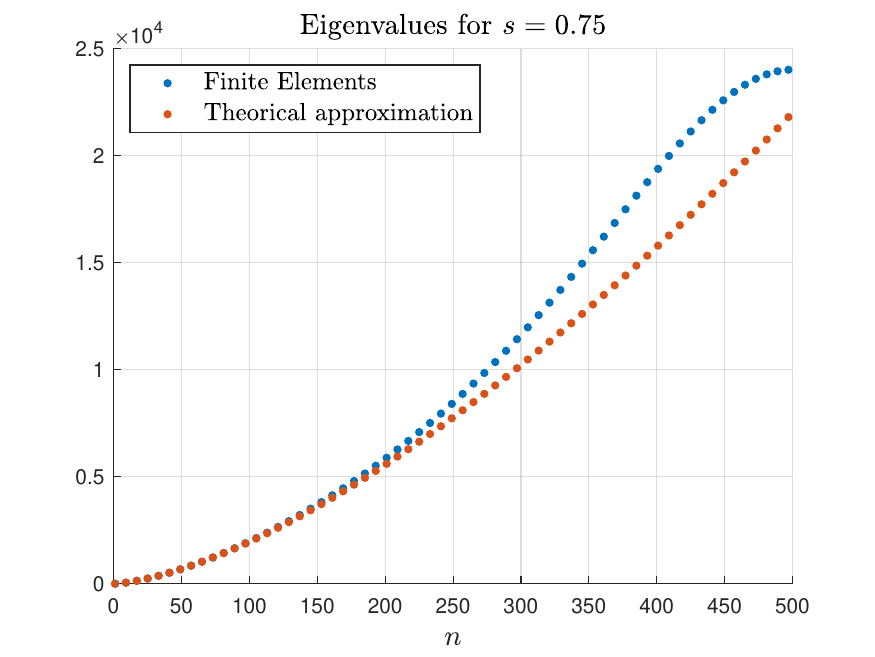}
\end{center}~\\[-10pt]
\caption{Eigenvalues estimated with FEM and approximation by \eqref{aproxeigenvalues}.}
\label{fig01}
\end{figure}

We note that the eigenvalues to be used will be up to \( e_n \) Fourier coefficients to be calculated or up to \( n^* = N / 5 \) terms in the partial sums. The error between the theoretical and numerical approximations depends on the number of nodes considered (\( N = 500 \)). By increasing to \( N = 750 \) and \( N = 1000 \), the finite element approximation of the first 500 eigenvalues improves substantially relative to the theoretical approximation, as shown in Figure \ref{fig023} for \( s = 0.25 \) and \( s = 0.75 \).
\begin{figure}[!h]
\begin{center}
\includegraphics[scale=0.5]{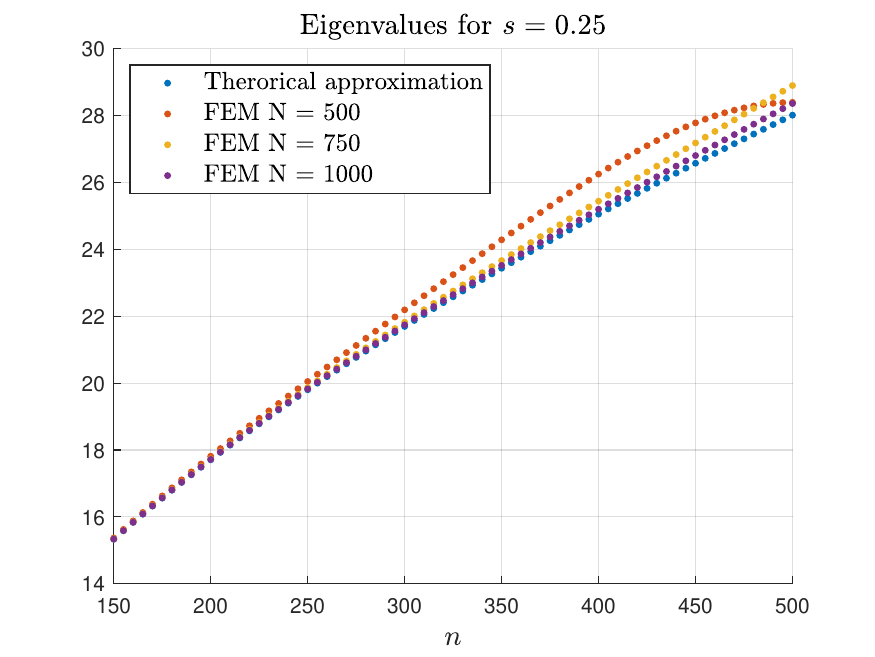}
\includegraphics[scale=0.5]{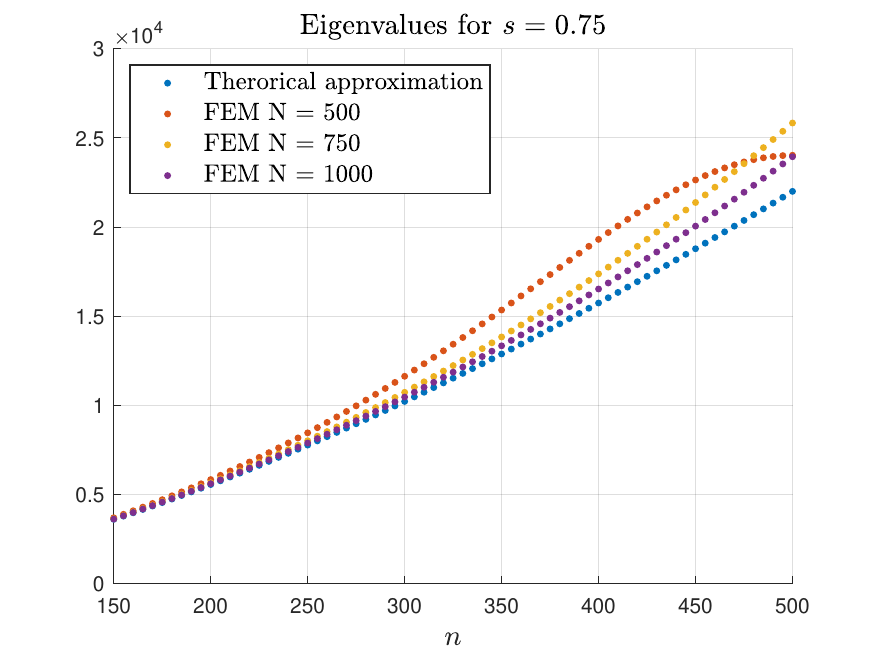}
\end{center}~\\[-10pt]
\caption{Eigenvalues according to the size of the partition.}
\label{fig023}
\end{figure}

Thus, to ensure the correct capture of the first 300 eigenvalues, all figures henceforth assume \( N = 1500 \), unless otherwise specified.

\subsection{Eigenfunction Approximation}

When computing the eigenfunctions, we recall that an asymptotic approximation for \( \varphi_n \) is provided in \cite{kwasnicki2012eigenvalues} for a sufficiently large \( n \). Although this work offers a more detailed expression, for simplicity, we compare our numerical results with the following approximation:
\begin{equation}
\tilde{\varphi}_n(x) = \sin \left( 2s \sqrt{\tilde{\lambda}_n} x + \frac{n \pi}{2} \right), \quad x \in \Omega,
\label{aproxeigenfunctions}
\end{equation}
which is illustrated in Figure \ref{fig05x}, for \( s = 0.25 \) and \( s = 0.75 \).
\begin{figure}[!h]
\begin{center}
\includegraphics[scale=0.5]{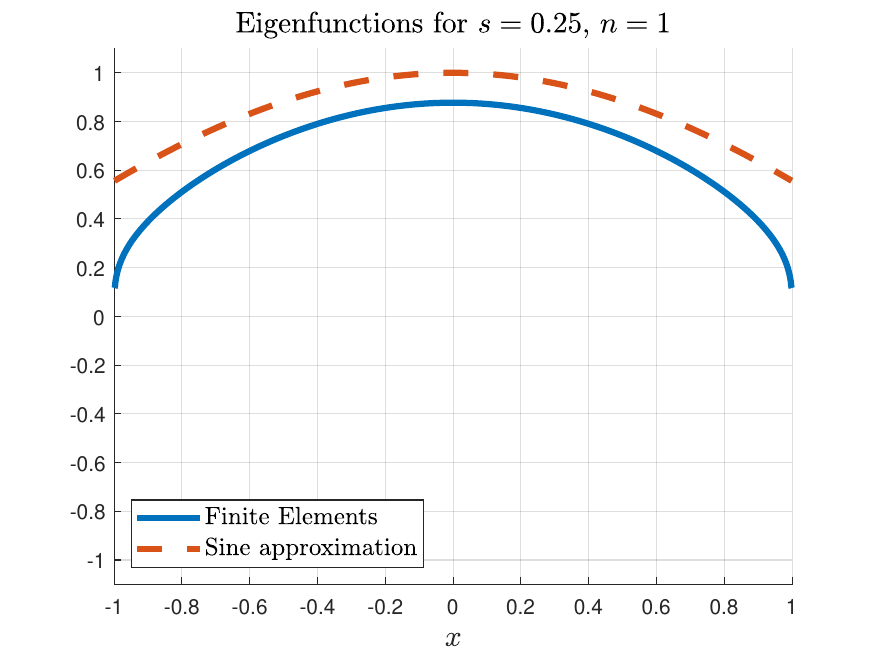}
\includegraphics[scale=0.5]{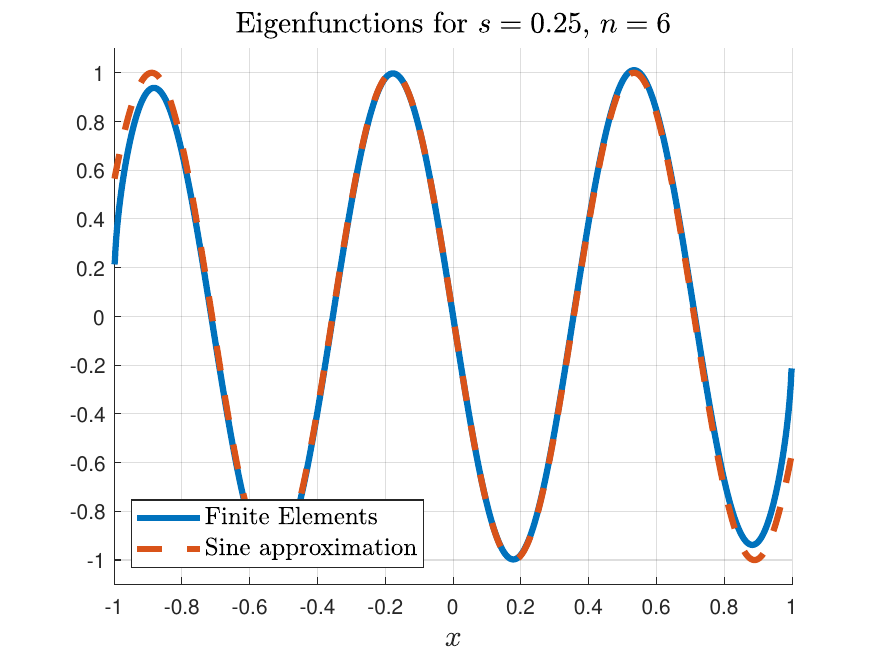}
\includegraphics[scale=0.5]{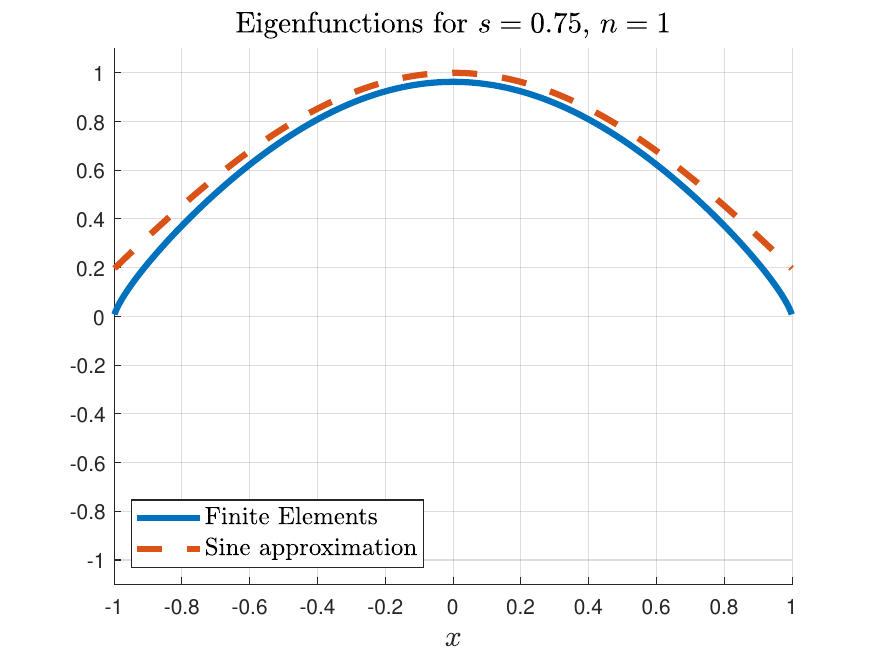}
\includegraphics[scale=0.5]{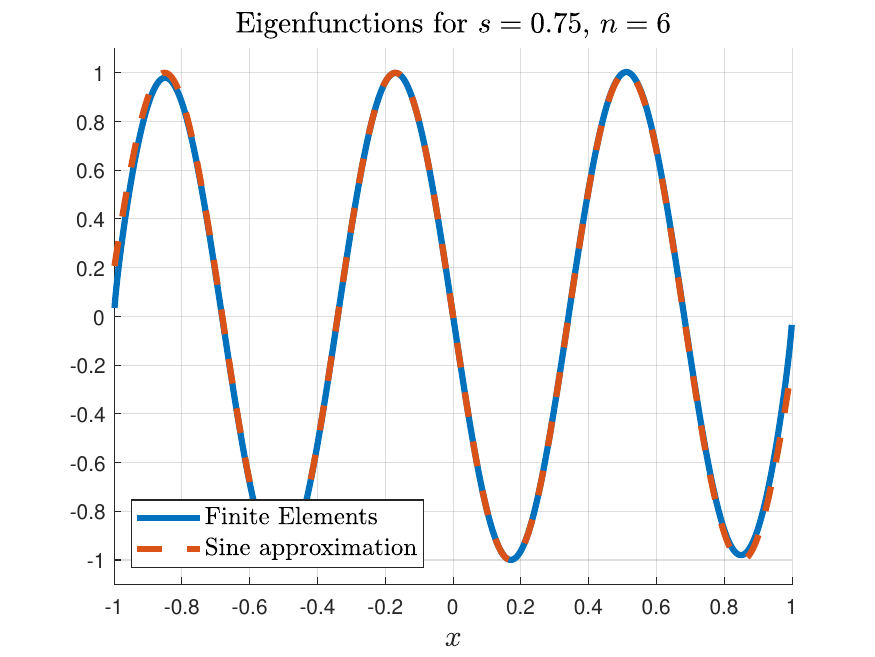}
\end{center}~\\[-10pt]
\caption{Eigenfunctions estimated with FEM and approximation by \eqref{aproxeigenfunctions}.}
\label{fig05x}
\end{figure}
This simpler approximation is designed to ensure a consistent sign for the function because the computed eigenvector may change its orientation when recalculated using different finite-element matrices.

\subsection{Configuration and Test Functions}

Next, we define the configuration used to solve the various problems. It is important to note that the finite element method for solving differential equations was validated in \cite{biccaricontrollability} using an explicit solution. However, in this study, we focused solely on visualizing the results obtained using this methodology.

The four functions defined below were selected to evaluate the performance of the numerical implementation, along with the numerical Finite Element solution to \eqref{eq5.1}, when \( \sigma(t) = e^t \) and \( s = 0.75 \):
\begin{align*}
f_1 &= \frac{5 (1 - x^2) (0.25 - x^2)(0.042 - x^2)}{2 (0.025 + x^2)}, \quad f_2 = (1 - x^2) \tan \left( \frac{\pi}{2.1} x \right),\\
f_3 &= \begin{cases} (5x - 1) \sin(20 \pi x) / 2 & \text{if } |x| < 0.2, \\ 0 & \text{otherwise,} \end{cases}\\
f_4 &= \begin{cases} -1 & \text{if } -0.95 < x \leq -0.8, \\ 1 & \text{if } 0.1 < x < 0.4, \\ 0 & \text{otherwise.} \end{cases}
\end{align*}

This selection aims to capture diverse behaviors:
\begin{enumerate}
\item Continuity: The first three functions are continuous.
\item Parity and Symmetry: \( f_1 \) is even, \( f_2 \) and \( f_3 \) are odd, and \( f_4 \) is asymmetric with respect to \( x = 0 \).
\item Frequency: \( f_1 \) and \( f_2 \) have dominant low-frequency components, whereas \( f_3 \) and \( f_4 \) contain high-frequency components.
\item Sign and Support: Only \( f_3 \) and \( f_4 \) are zero over more than half of the interval, and all exhibit a sign change. Moreover, \( f_4 \) has significant support outside the observation interval, as we will be defined later.
\end{enumerate}

Unless otherwise stated, the examples assume \( s = 0.75 \), \( T = 1 \), and mesh sizes \( M = 10000 \) and \( f_M = 100 \). Additionally, we evaluated the solution for different functions \( \sigma \).

\subsection{Zero Control}

To compute Fourier coefficients, we must first construct a family of null controls. This requires computing the matrices \( \Lambda_N \) defined in \eqref{eq5.3} for each \( \tau = \tau_\ell \), \( \ell = 1, 2, \ldots, f_M \), and the right-hand side of equation \( \Phi_0(\cdot, 0) \) for each initial condition \( \varphi_n \).

To reduce the computation time, we opted to compute \( \psi_j(\cdot, 0) \), \( j = 1, \ldots, N \) using partial sums instead of solving numerous differential equations on an \( N \times M \) space-time mesh. For instance, for \( j = N / 2 \), Figure \ref{fig18} shows that \( \psi_j(\cdot, 0) \) presents no significant differences when computed using partial sums versus the full finite-element differential equation.
\begin{figure}[!h]
\begin{center}
\includegraphics[scale=0.6]{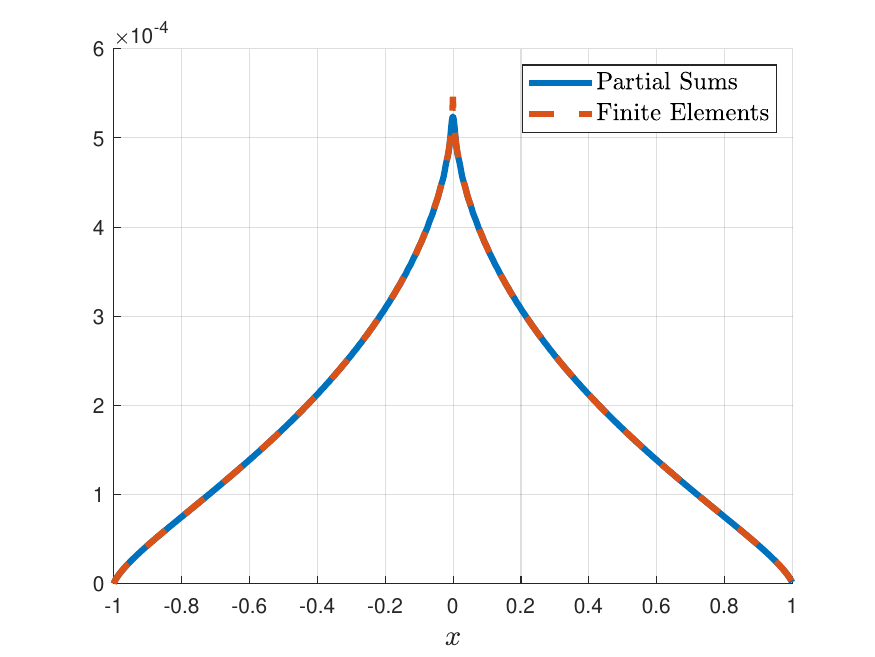}
\end{center}~\\[-10pt]
\caption{Approximation of final state $\Phi_j(x,0)$ with FEM and Partial Sums.}
\label{fig18}
\end{figure}

Similarly, \( \Phi_0(\cdot, 0) \) was computed using partial sums (an exact formula in this case), as illustrated in Figure \ref{fig19}. Solving system \eqref{eq5.4} provides an approximation of the controlled state \( \phi(\cdot, 0) \approx \varepsilon^{-1} w_0 \), and consequently, the control $ h$ using the projection formula \eqref{fop}.
\begin{figure}[!h]
\begin{center}
\includegraphics[scale=0.6]{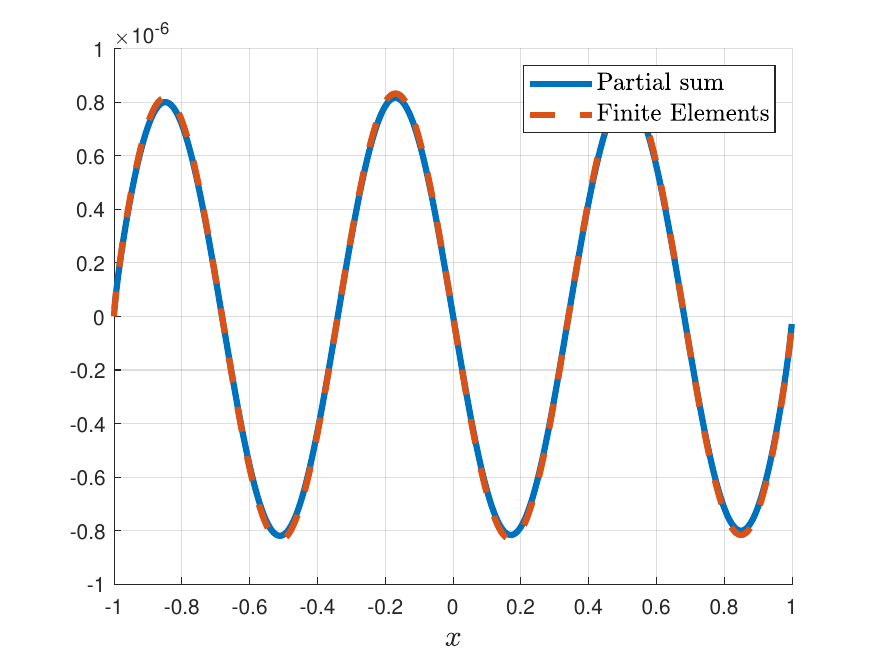}
\end{center}~\\[-10pt]
\caption{Approximation of final state $\Phi_0(x,0)$ with FEM and exact formula.}
\label{fig19}
\end{figure}

If we perform this calculation starting from $\tau=0.5$ with $\phi(\cdot,0)=\varphi_6$, in the observation interval $\omega= (-0.75,0.75)$, together with a penalty parameter $\varepsilon=10^4$, the control only acts when $t$ is close to zero, as illustrated in Figure \ref{fig201}, where $\omega$ is marked as the area delimited with dashed lines. This behavior exhibited by the control is natural because its norm is also involved in the optimization problem \eqref{funct}-\eqref{control}, meaning that the control allows the solution to ``lose energy'' before it starts acting. 
\begin{figure}[!h]
\begin{center}
\includegraphics[scale=0.5]{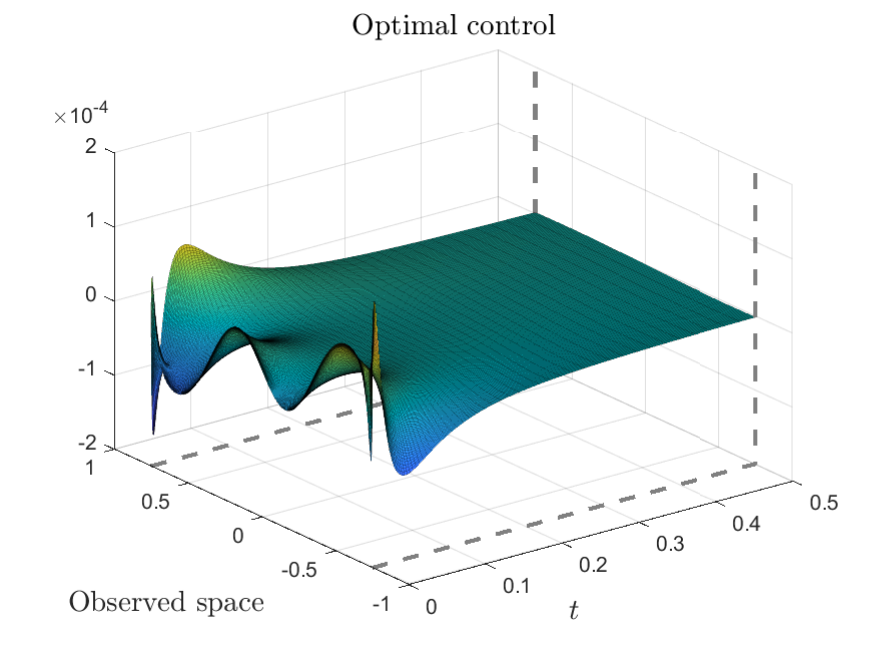}
\includegraphics[scale=0.5]{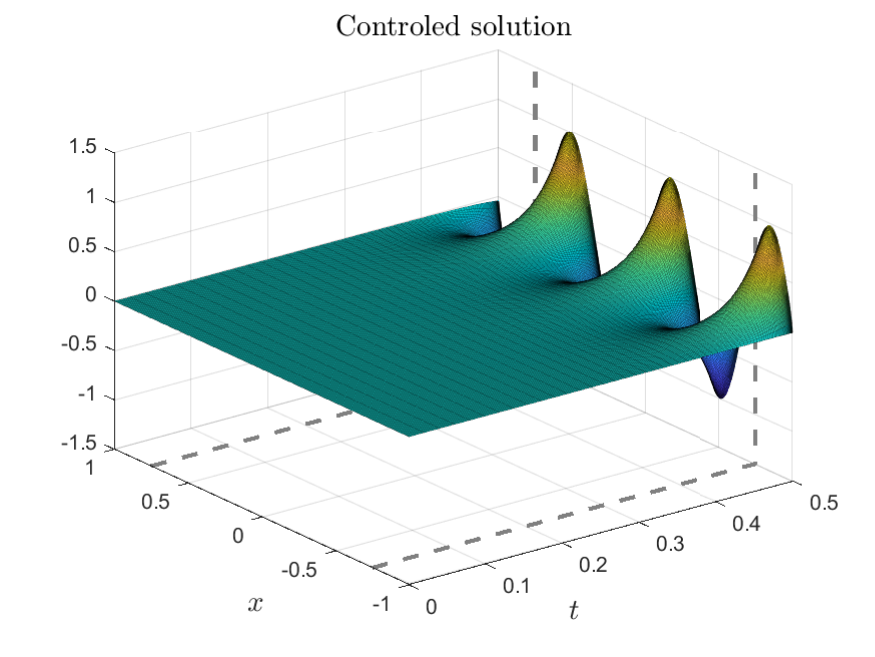}
\end{center}~\\[-10pt]
\caption{Control and controlled solution starting at $\tau=0.5$.}
\label{fig201}
\end{figure}

However, if the time in which the solution is controlled is less,
the control starts to act immediately to reduce the state at $t=0$ as much as possible, as shown in Figure \ref{fig234}.
\begin{figure}[!h]
\begin{center}
\includegraphics[scale=0.5]{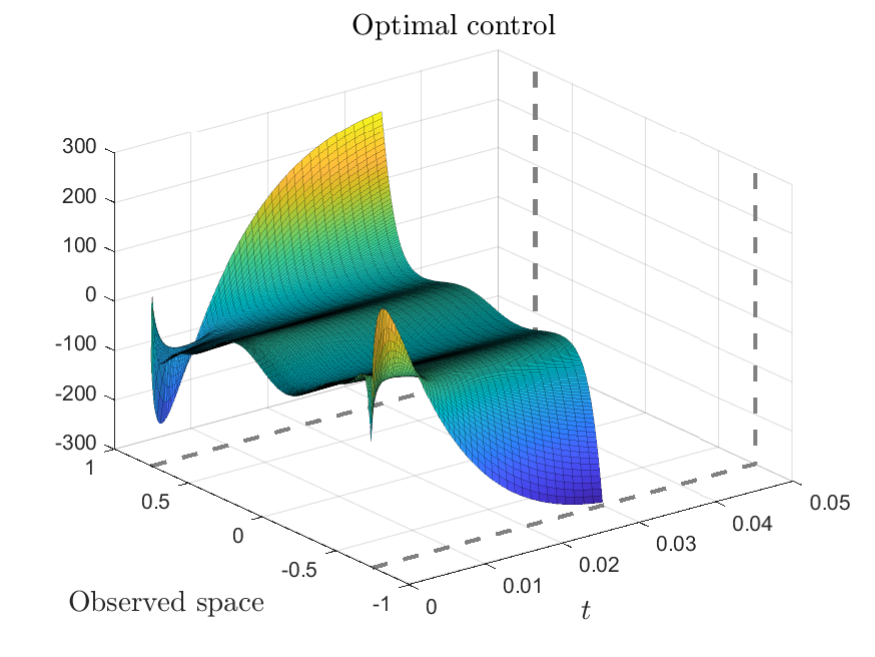}
\includegraphics[scale=0.5]{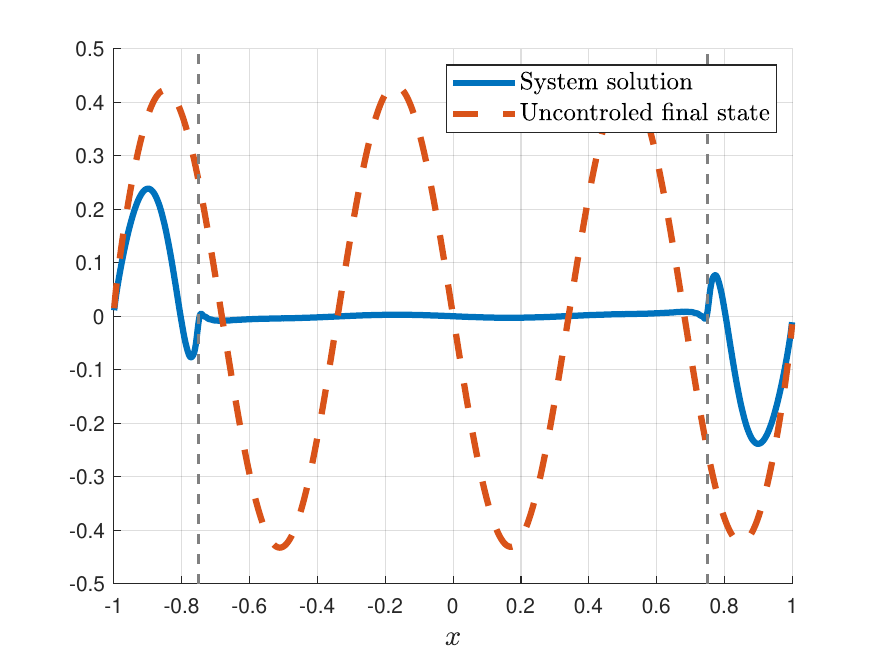}
\end{center}~\\[-10pt]
\caption{Control and controlled final state starting at $\tau=0.03$.}
\label{fig234}
\end{figure}

It is necessary to evaluate whether the control obtained from $w_0$, effectively causes the solution of the control problem to satisfy $\phi(\cdot,0)=\varepsilon^{-1}w_0$,  when calculating $p_{\varepsilon}$ and $h_{\varepsilon}$ of the optimal system \eqref{control-3}-\eqref{dual-1}, the result is as expected. There are notable differences if it is calculated with finite elements instead of partial sums, when $\tau$ is sufficiently large.
\begin{figure}[!h]
\begin{center}
\includegraphics[scale=0.5]{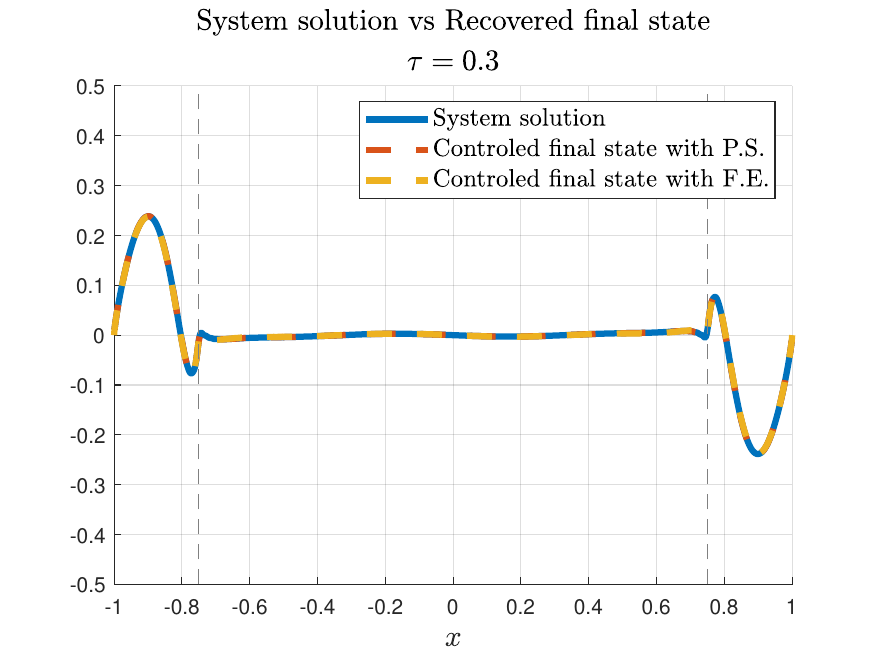}
\includegraphics[scale=0.5]{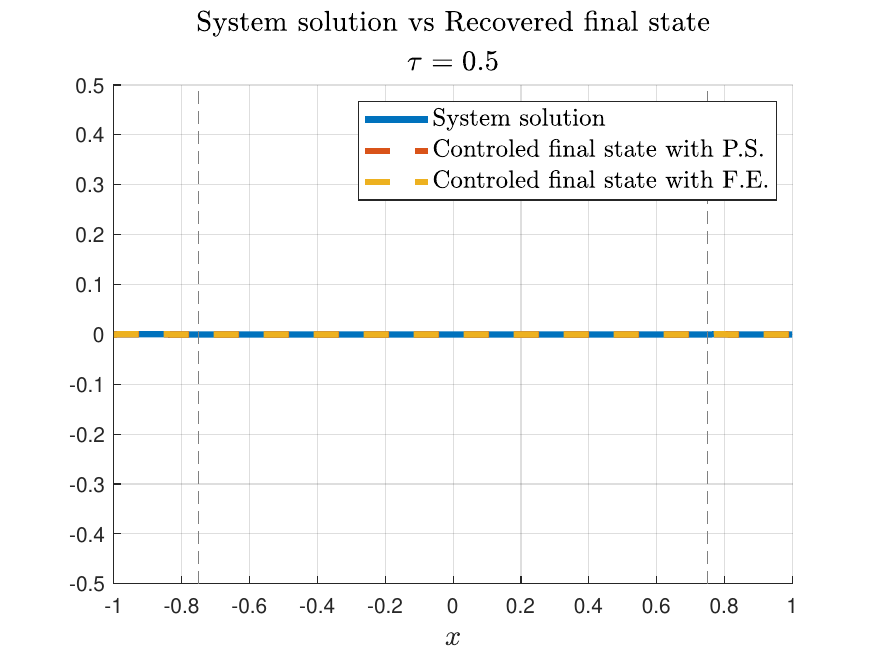}
\end{center}~\\[-10pt]
\caption{Recovery of $\phi(\cdot,0)=\varepsilon^{-1} w_0$ solving fractional heat equation with FEM and Partial Sums.}
\label{fig256}
\end{figure}

To conclude the control problem, it is necessary to evaluate the behavior of $h_\varepsilon$ and $\phi(\cdot,0)$ when $\varepsilon\to\infty$. As discussed in Section \ref{sec-4}, it is to be expected that $\|\phi(\cdot,0)\|_{L^2(\Omega)}\to 0$. However, in \cite{biccari2019controllability}, we require an increasingly finer discretization when $\varepsilon$ increases. That said, using $N = 5000$ and $\tau=0.03$ we can appreciate in Figure \ref{fig278} how the norm decreases, but requires $\varepsilon$ to diverge faster and faster. Moreover, together with the convergence of $\phi(\cdot,0)$, we can appreciate the growth in the $L^2(\omega\times(0,\tau))$ and $L^\infty (\omega\times(0,\tau))$ control norms.
\begin{figure}[!h]
\begin{center}
\includegraphics[scale=0.5]{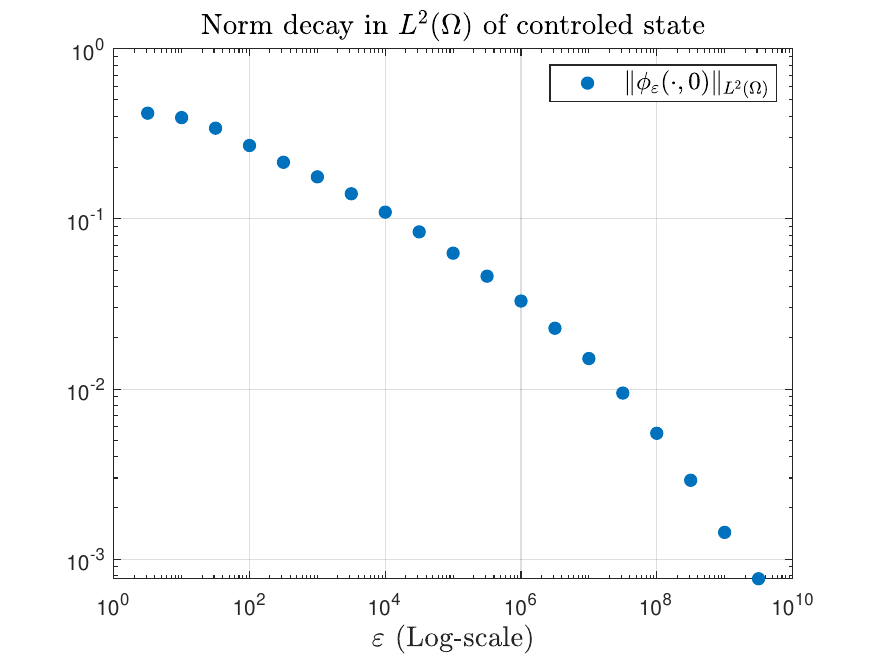}
\includegraphics[scale=0.5]{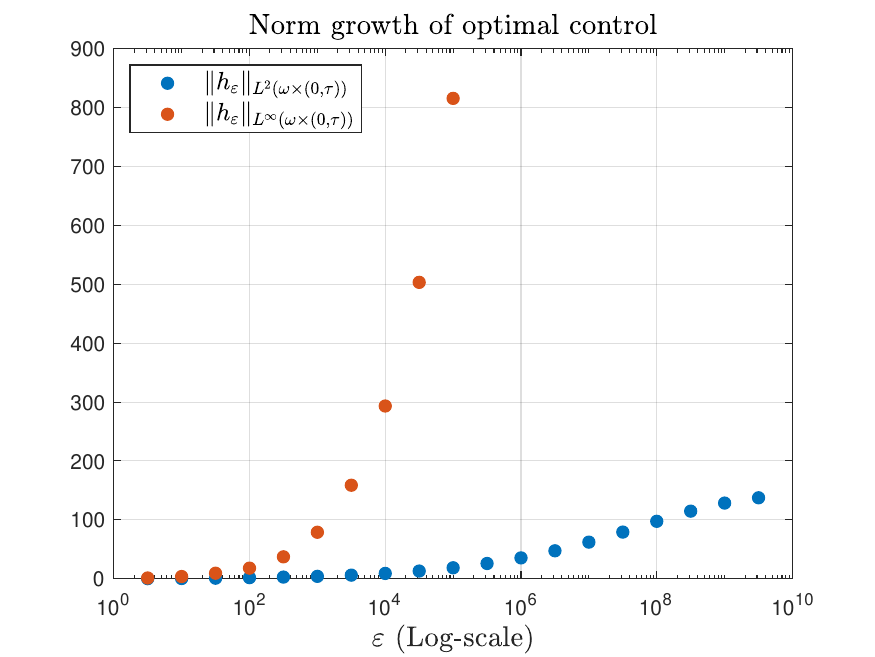}
\end{center}~\\[-10pt]
\caption{Norm trends of the control and the controlled final state.}
\label{fig278}
\end{figure}

Considering the different variables involved in the control, for a fixed $\widetilde{M}$ it has been decided to choose the optimization parameter $\varepsilon$ was chosen as follows:
\begin{align*}
\widetilde{\varepsilon}=\inf\left\{\varepsilon>0: \ \left(\max_{n\leq \widetilde{n}} \|\phi_{\varepsilon,n}^{(\tau_1)}(\cdot,0)\|_{L^2(\omega)} \right) <m(\omega) \times 1\% \right\},
\end{align*}
where $\phi_{\varepsilon,n}^{(\tau_1)}$ are the controlled solutions of the optimal control problem with control $h_\varepsilon^{(\tau_1)}$ and the conditions $\phi_{\varepsilon,n}^{(\tau_1)}(\cdot,\tau_1)=\varphi_n$ and $m(\omega)$ is the measure of the observation set $\omega$. This for the configuration used thus far and $\omega= (-0,75,0,75)$ results that $10^{3,9}<\widetilde{\varepsilon} \leq 10^4.$

\subsection{Volterra Equation}

The next step involves solving Volterra equations. Figure \ref{fig301} shows how the control solutions presented in Figures \ref{fig201} and \ref{fig234} affect the integral solution \( \theta(\tau)_\varepsilon \) for a small \( \tau \) and \( \tau = 0.5 \), indicating that its impact is minimal when \( \tau \) is sufficiently large.
\begin{figure}[!h]
\begin{center}
\includegraphics[scale=0.5]{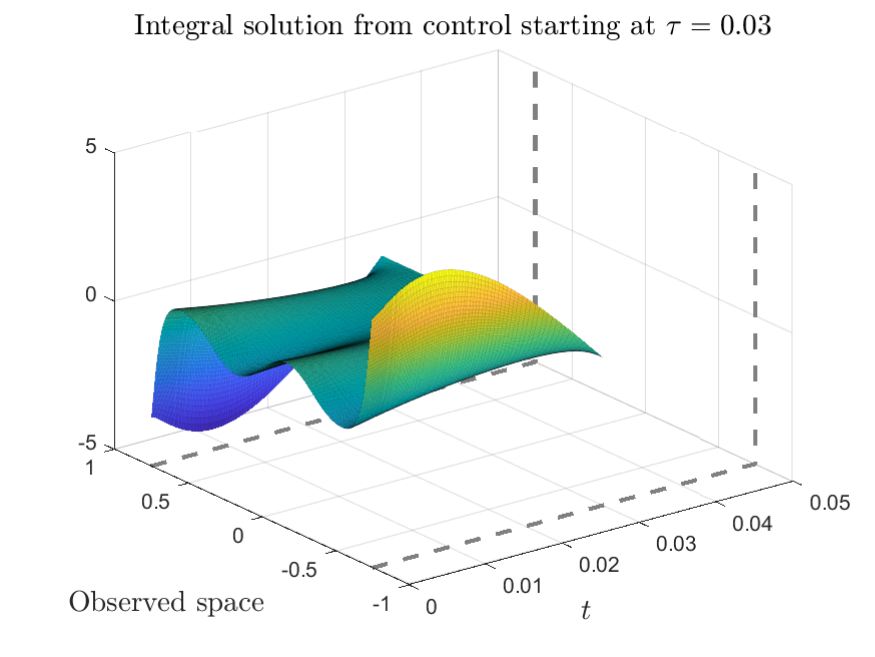}
\includegraphics[scale=0.5]{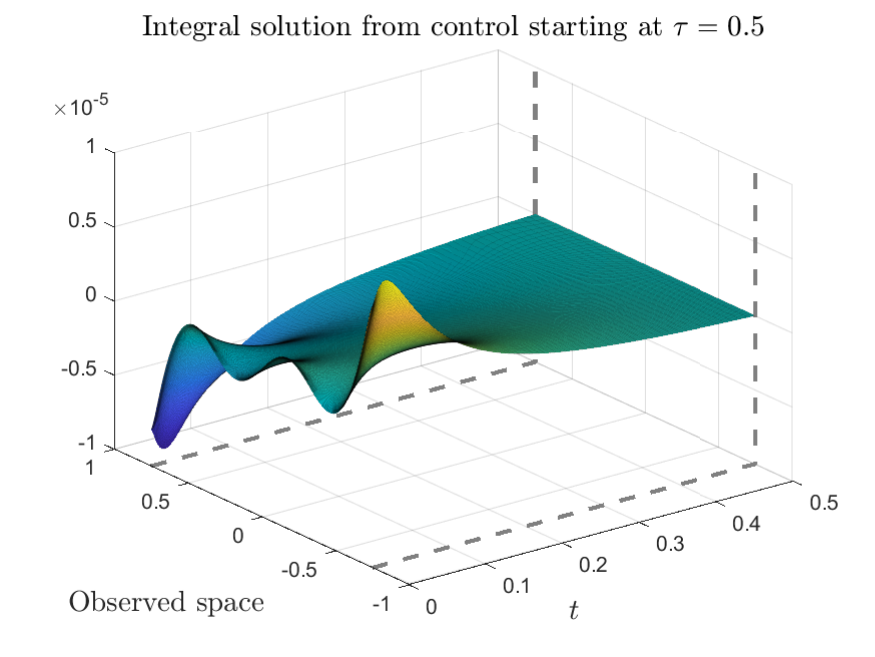}
\end{center}~\\[-10pt]
\caption{Solution to Volterra's equation.}
\label{fig301}
\end{figure}

In addition, as the controls reach higher values when \( \varepsilon \to +\infty \), Figure \ref{fig31x} shows that the integral solution exhibits a similar trend, albeit less drastic.
\begin{figure}[!h]
\begin{center}
\includegraphics[scale=0.5]{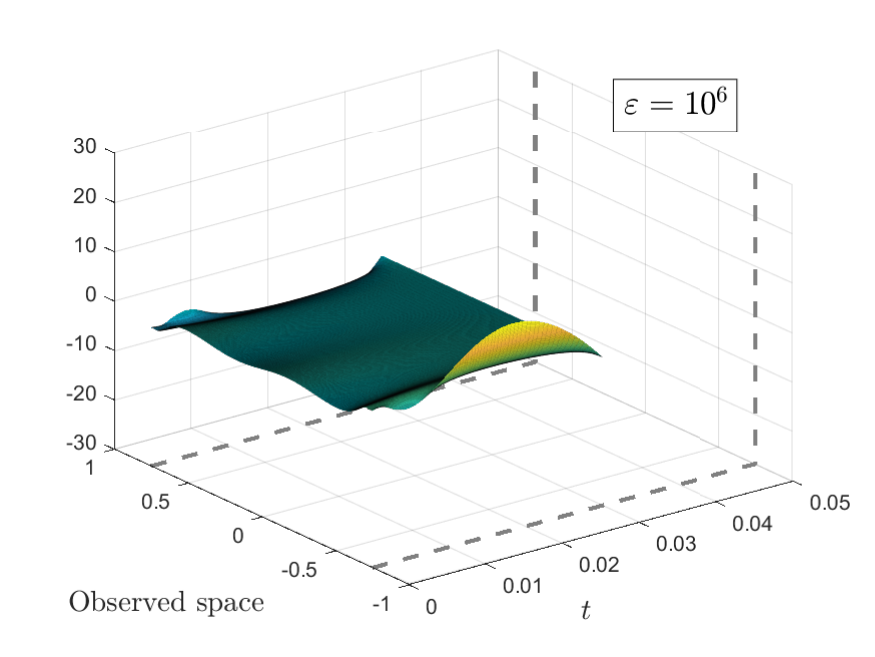}
\includegraphics[scale=0.5]{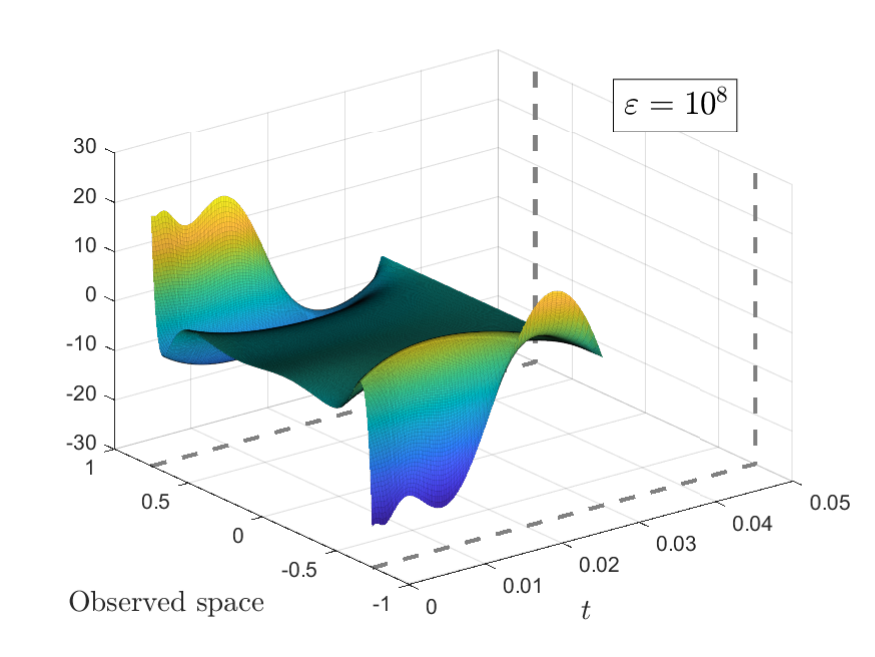}
\end{center}~\\[-10pt]
\caption{Integral solutions for increasing $\varepsilon$ and $\tau=0.03$.}
\label{fig31x}
\end{figure}

\subsection{Results}

After analyzing each calculation stage, we computed the Fourier coefficients of $ f $ using the formulas presented in this article. However, using the configuration: $ N = 1500$,  $ M = 10000$, and $f_M = 100$  requires significant RAM. Therefore, the number of points was reduced to \( N = 500 \) and \( M = 1000 \), allowing us to approximate the first 50 Fourier coefficients much faster and with good results, as shown in Figure \ref{fig3x}, where the first formula from Theorem \ref{th1} and an optimization parameter \( \varepsilon = 10^4 \) were used.
\begin{figure}[!pt]
\begin{center}
\includegraphics[scale=0.4]{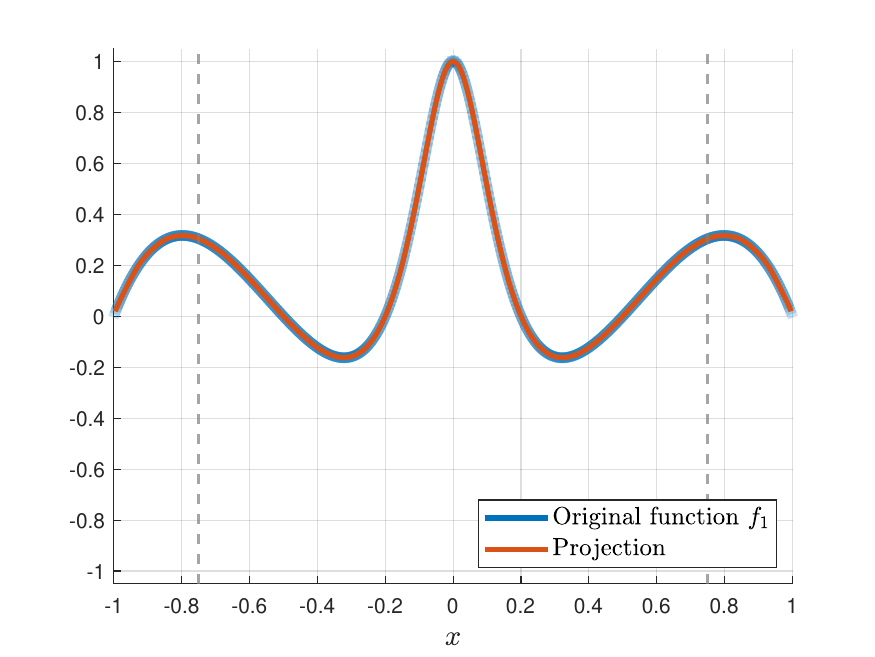}
\includegraphics[scale=0.4]{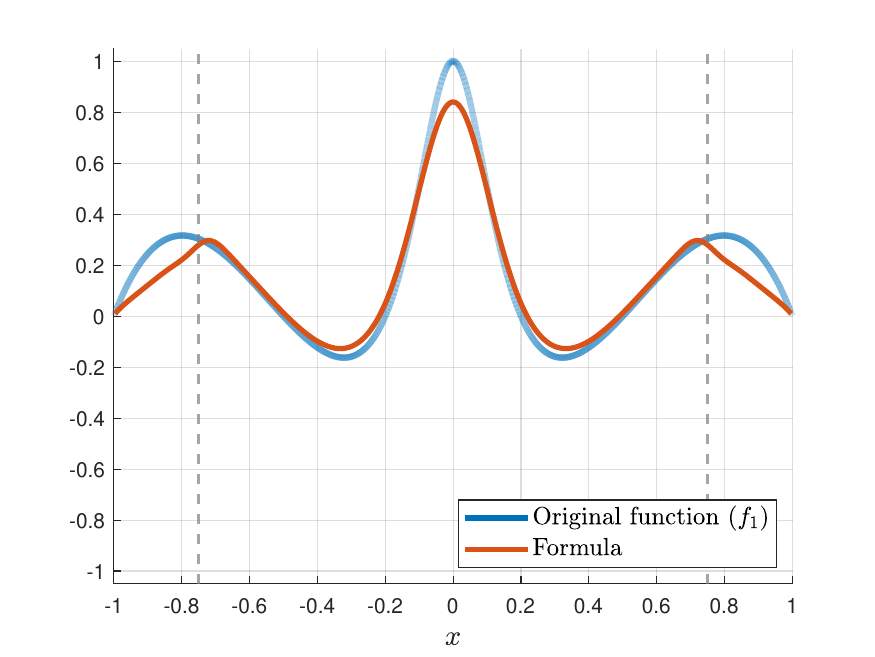}
\includegraphics[scale=0.4]{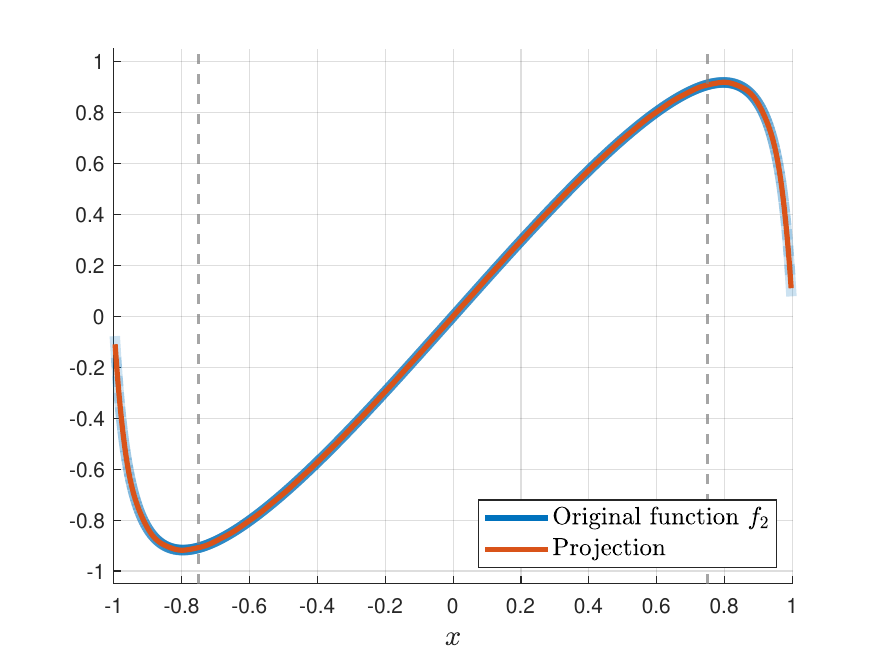}
\includegraphics[scale=0.4]{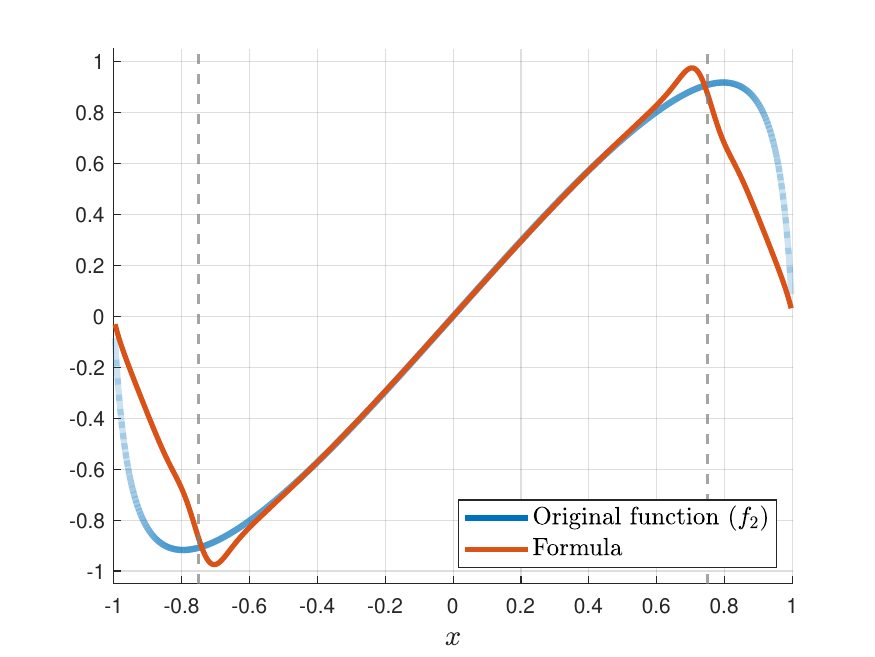}
\includegraphics[scale=0.4]{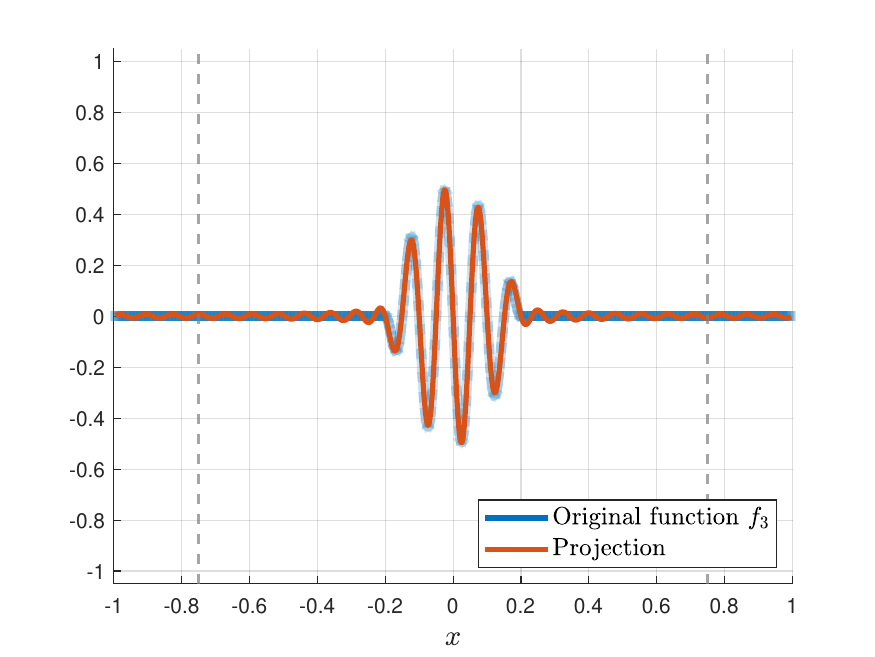}
\includegraphics[scale=0.4]{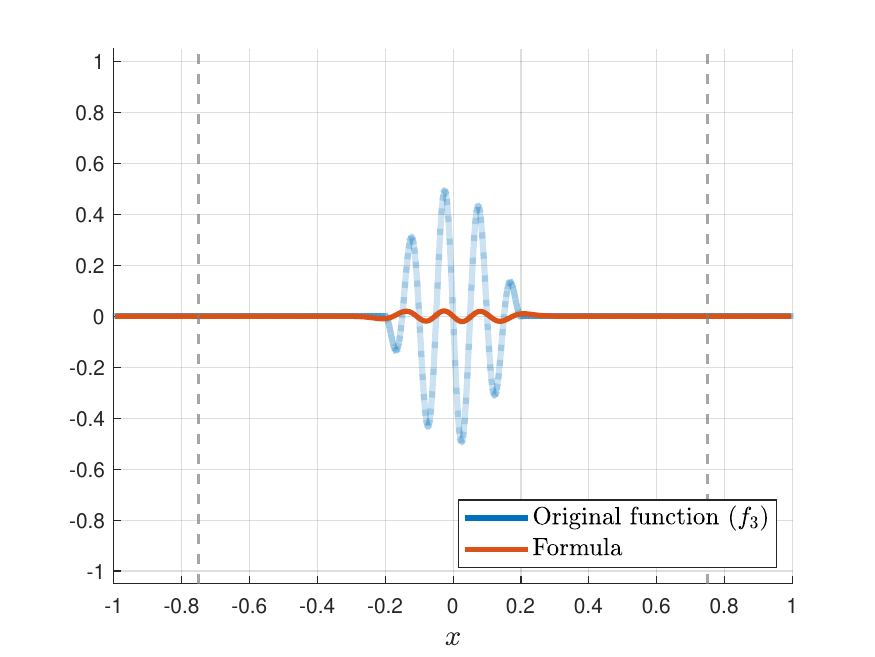}
\includegraphics[scale=0.4]{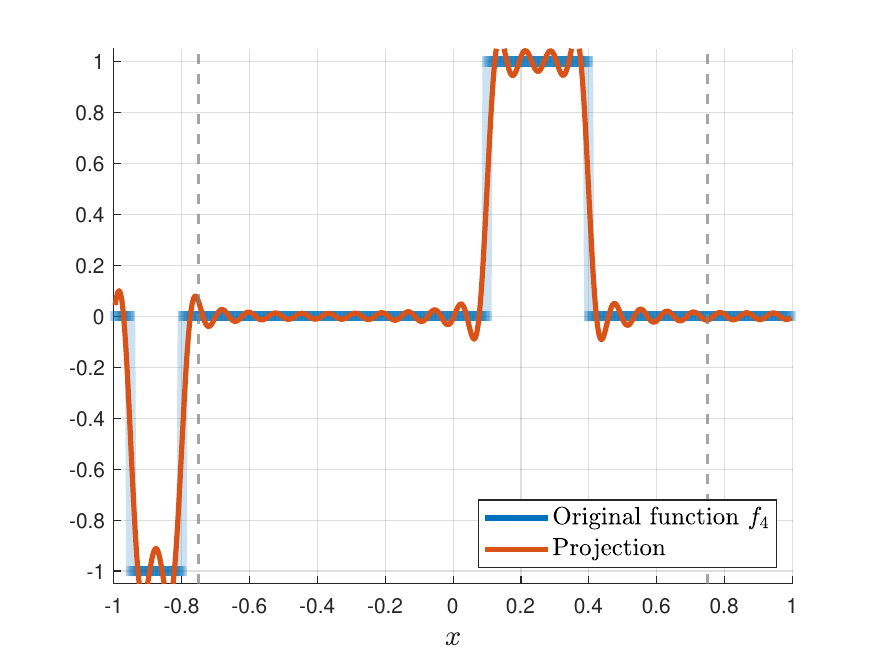}
\includegraphics[scale=0.4]{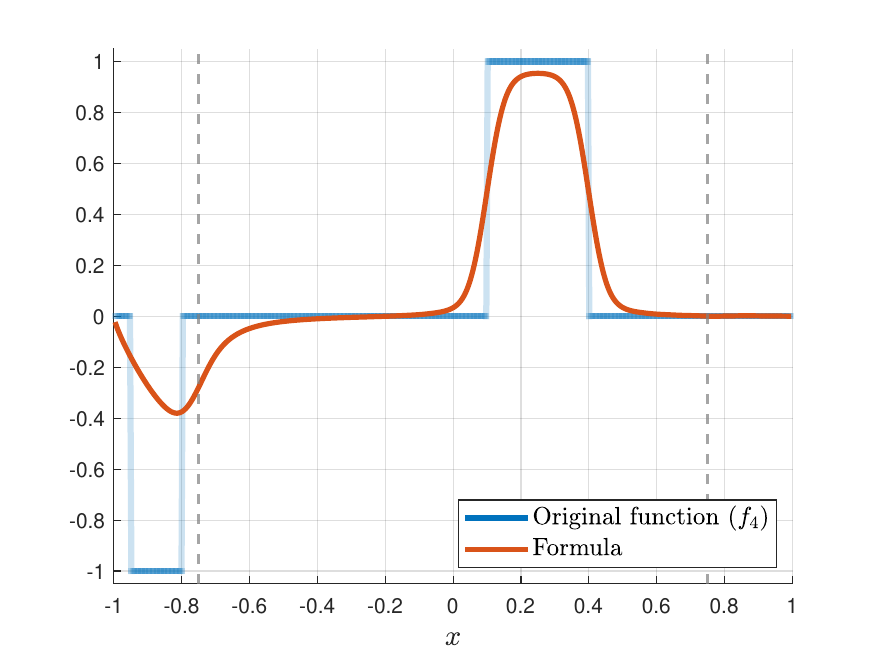}
\end{center}~\\[-10pt]
\caption{Reconstruction of the spatial part of the source, from the observation interval delimited with the dashed lines, where the left column is $f$ and its projection with 50 Fourier coefficients and in the right column the recovered projection.}
\label{fig3x}
\end{figure}

In each case, we observed that the support and shape of the function were captured accurately. Despite more errors in the third function, all the cases correctly approximated the support and shape of the original function.

It is interesting to note that, despite correctly recovering the support and shape of the function, without oscillating too much (as with $f_3$ and $f_4$), the Fourier coefficients are not captured well. For example, in Figure \ref{fig41}, we see that, for $f_2$, coefficients 8, 9, 10, and 11 have opposite signs.
\begin{figure}[!h]
\begin{center}
\includegraphics[scale=0.5]{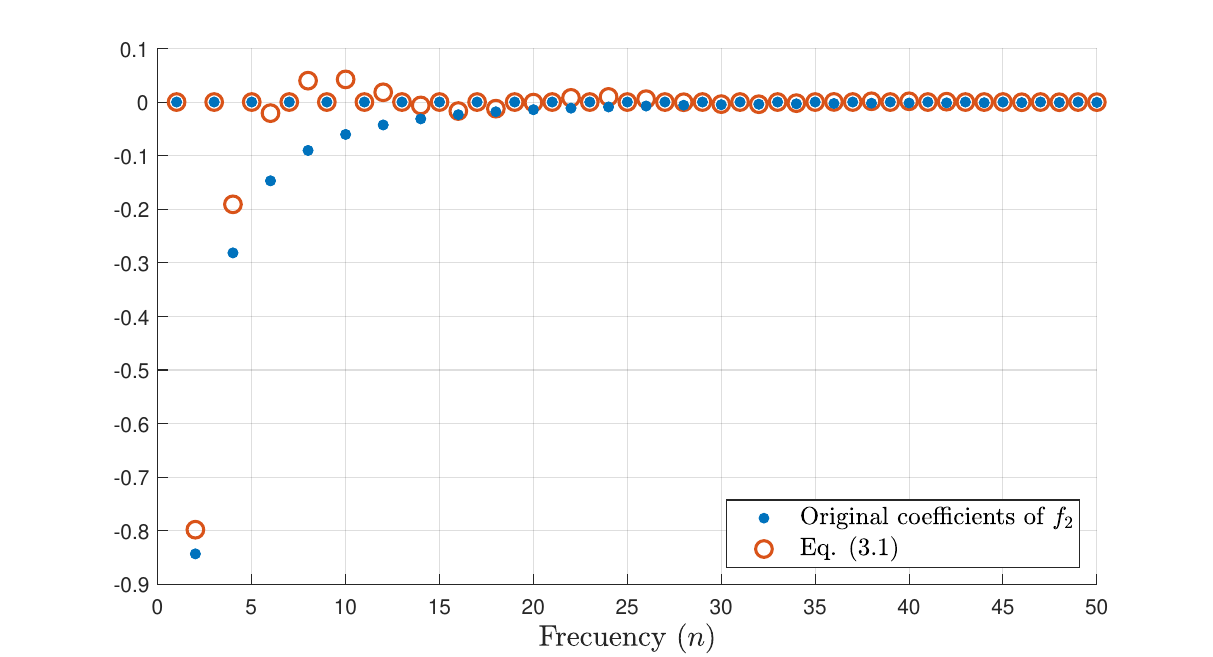}
\end{center}~\\[-10pt]
\caption{Fourier coefficients of $f_2$.}
\label{fig41}
\end{figure}

For $f_4$, this sign change is repeated in some coefficients, and in Figure \ref{fig4023}, we see how some high frequencies approach zero, when in fact this is not the case. Despite this, for $f_1$ and $f_2$ the calculated coefficients retain their signs and are distinguishably non-zero if the original coefficients are.
\begin{figure}[!pt]
    \centering
    \begin{subfigure}{0.48\textwidth}
        \centering
        \includegraphics[width=\linewidth]{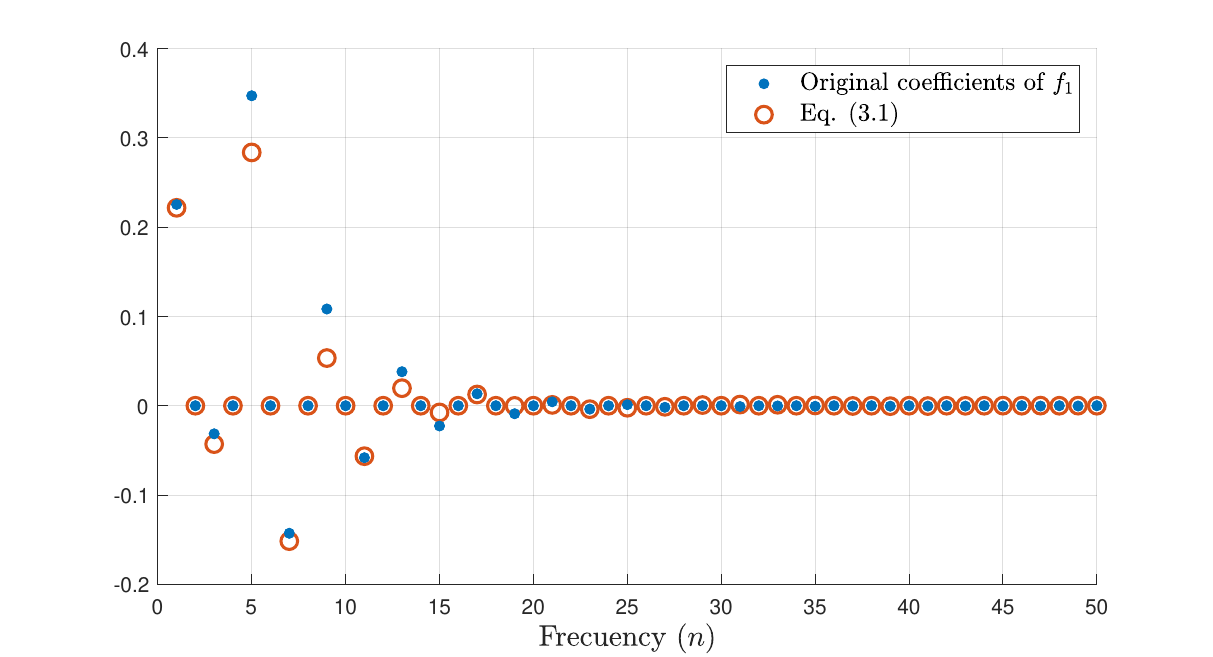}
        \caption{Fourier coefficients of $f_1$.}
        \label{fig:f1_coeffs}
    \end{subfigure}
    \hfill
    \begin{subfigure}{0.48\textwidth}
        \centering
        \includegraphics[width=\linewidth]{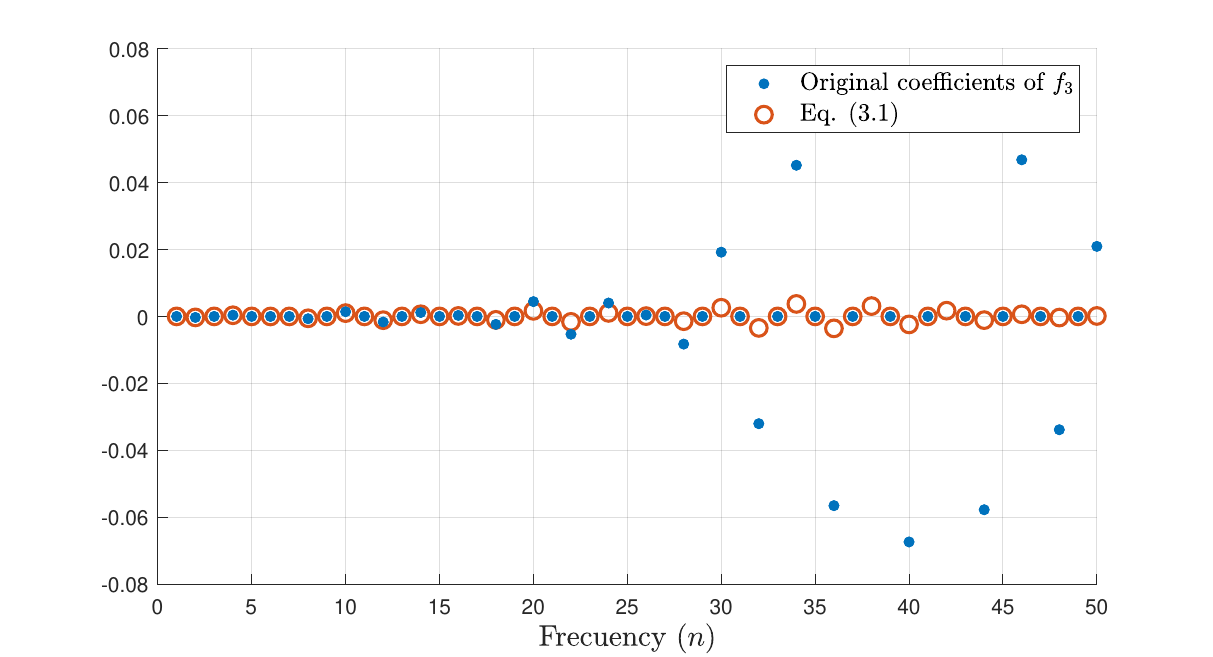}
        \caption{Fourier coefficients of $f_3$.}
        \label{fig:f3_coeffs}
    \end{subfigure}
    \\[10pt] 
    \begin{subfigure}{0.48\textwidth}
        \centering
        \includegraphics[width=\linewidth]{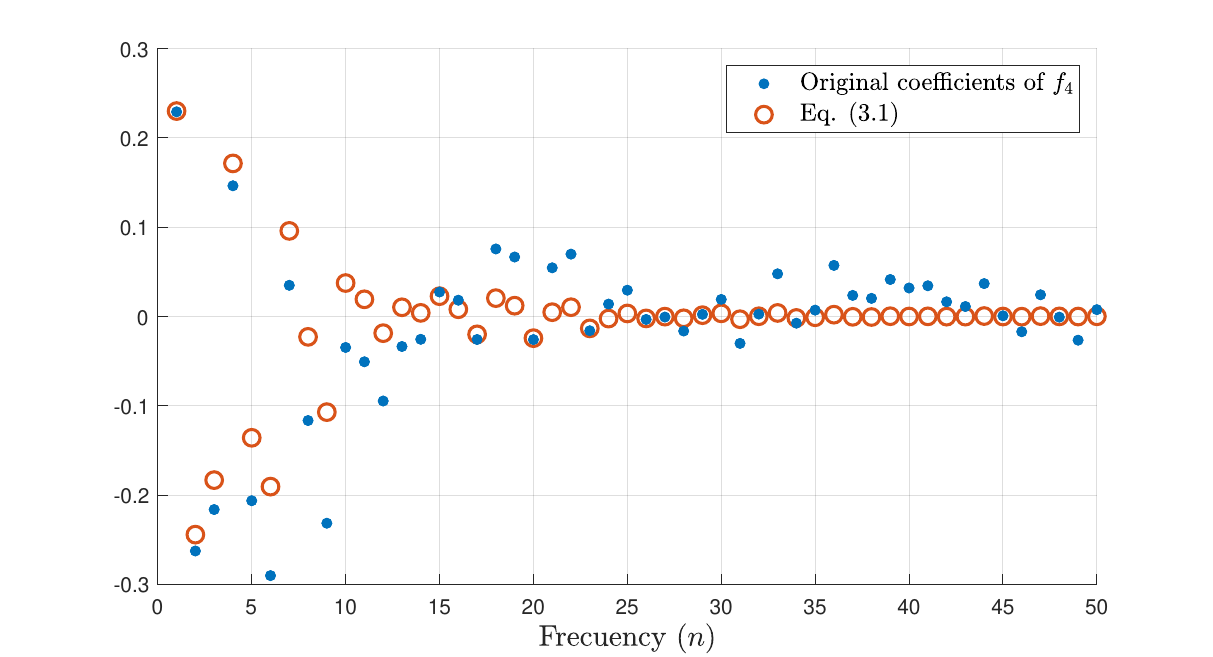}
        \caption{Fourier coefficients of $f_4$.}
        \label{fig:f4_coeffs}
    \end{subfigure}
      \caption{Comparison of the exact and reconstructed Fourier coefficients for the source functions (a) $f_1$, (b) $f_3$, and (c) $f_4$. The charts for $f_1$ and $f_4$ show good agreement in sign and magnitude for low frequencies, while the chart for $f_3$ shows more significant errors in capturing the correct coefficients.}
    \label{fig4023}
\end{figure}

When comparing the formulas of Theorems \ref{th1} and \ref{th2}, in Figure \ref{fig:all_fourier_coeffs}, we see how the same factors are recovered (with minimal differences in approximation). However, it is necessary to emphasize that the factor $c_n$ of Theorem \ref{th2} has been calculated not with the same expression, but after a change in the variable
\begin{align*}
c_n=e^{-\lambda_n T}\sigma(0)+\int_0^T e^{\lambda_n(s-T)}\sigma'(s)ds.
\end{align*}
\begin{figure}[!h]
    \centering
    \begin{subfigure}{0.48\textwidth}
        \centering
        \includegraphics[width=\linewidth]{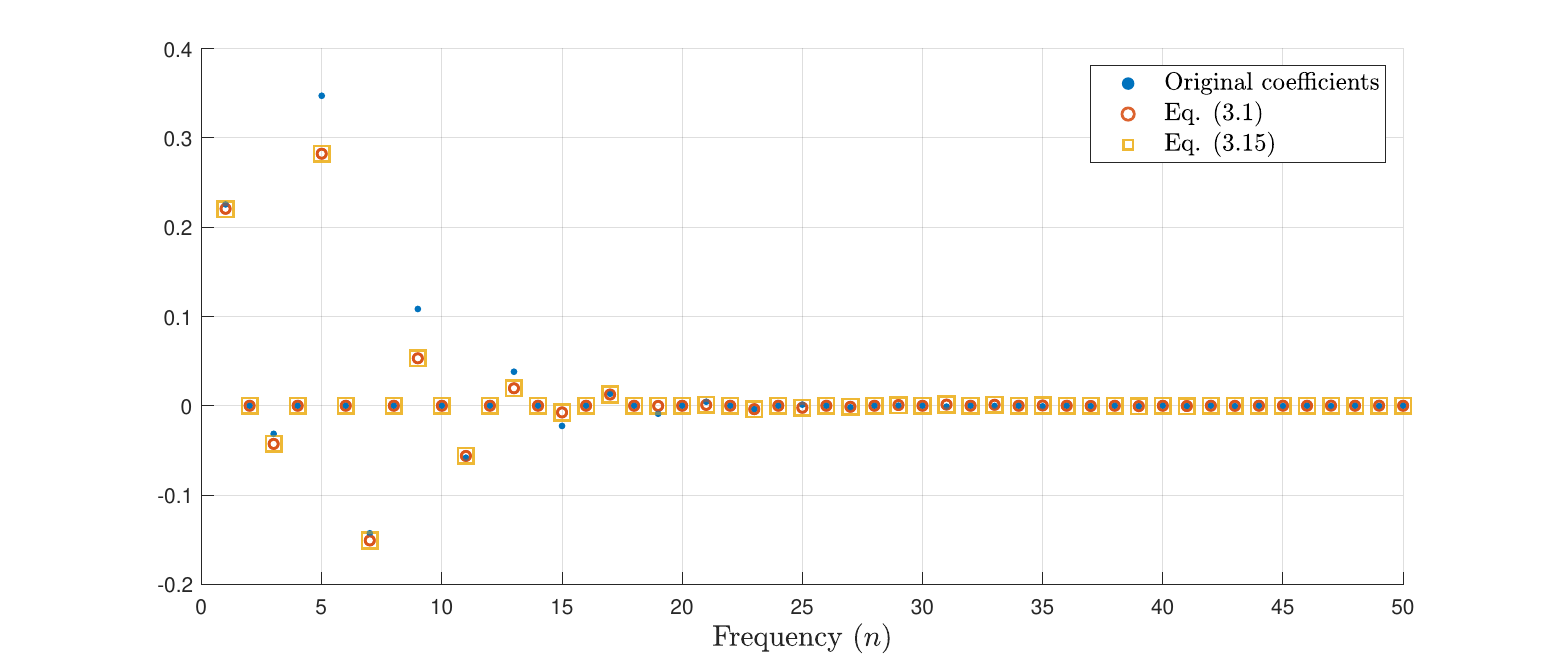}
        \caption{Fourier coefficients of $f_1$.}
        \label{fig:f1_coeffs_2}
    \end{subfigure}
    \hfill
    \begin{subfigure}{0.48\textwidth}
        \centering
        \includegraphics[width=\linewidth]{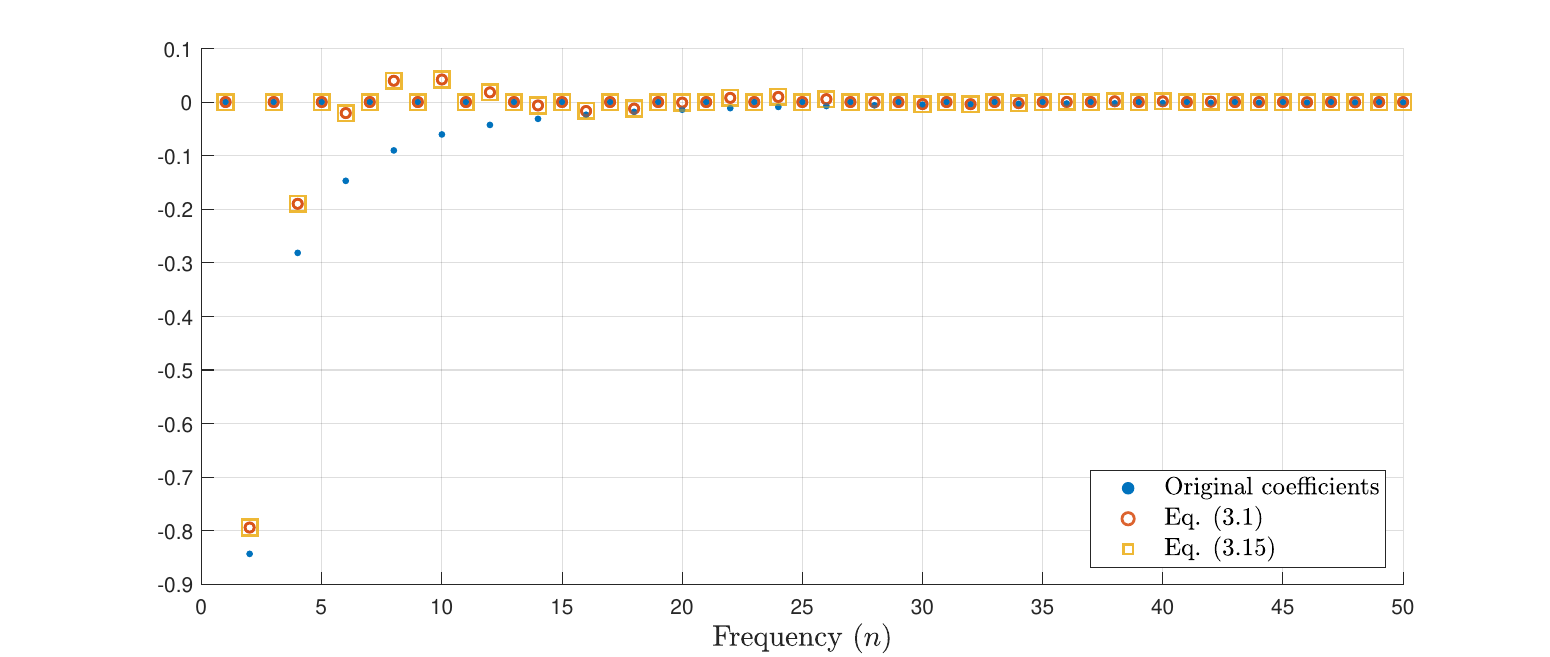}
        \caption{Fourier coefficients of $f_2$.}
        \label{fig:f2_coeffs_2}
    \end{subfigure}
    
    \vspace{0.5cm} 
    \begin{subfigure}{0.48\textwidth}
        \centering
        \includegraphics[width=\linewidth]{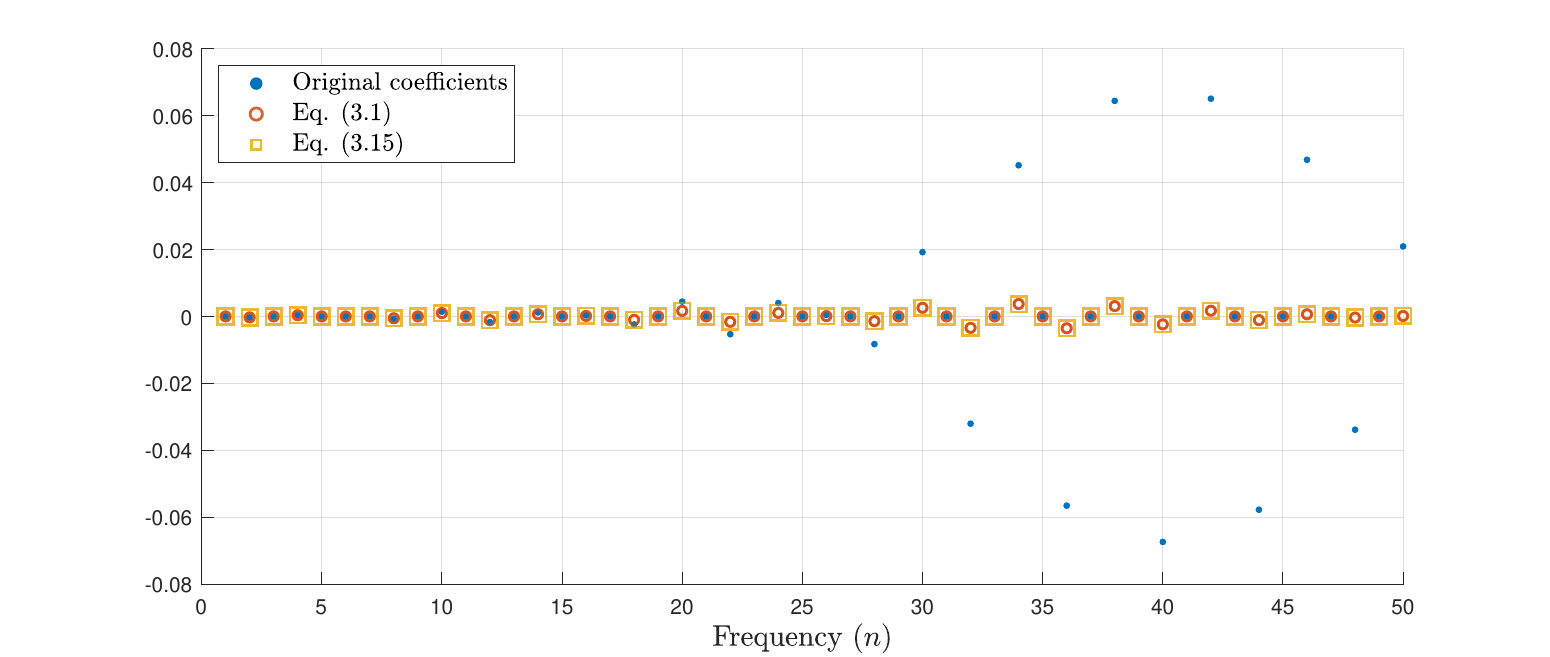}
        \caption{Fourier coefficients of $f_3$.}
        \label{fig:f3_coeffs_2}
    \end{subfigure}
    
    \caption{Comparison of the exact and reconstructed Fourier coefficients for the source functions (a) $f_1$, (b) $f_2$, and (c) $f_3$. This arrangement provides sufficient width for each chart to clearly display the coefficient indices and values, facilitating comparison of the reconstruction accuracy across different function types.}
    \label{fig:all_fourier_coeffs}
\end{figure}

Finally, we analyzed the recovery of the function when using different functions $\sigma$, which are illustrated in Figure \ref{fig70} together with their derivatives:
\begin{align*}
\sigma_1(t)=1, \quad \sigma_2(t)=\cos(10t),	\quad \sigma_3(t)=(1-t)^2.
\end{align*}
\begin{figure}[!h]
\begin{center}
\includegraphics[scale=0.32]{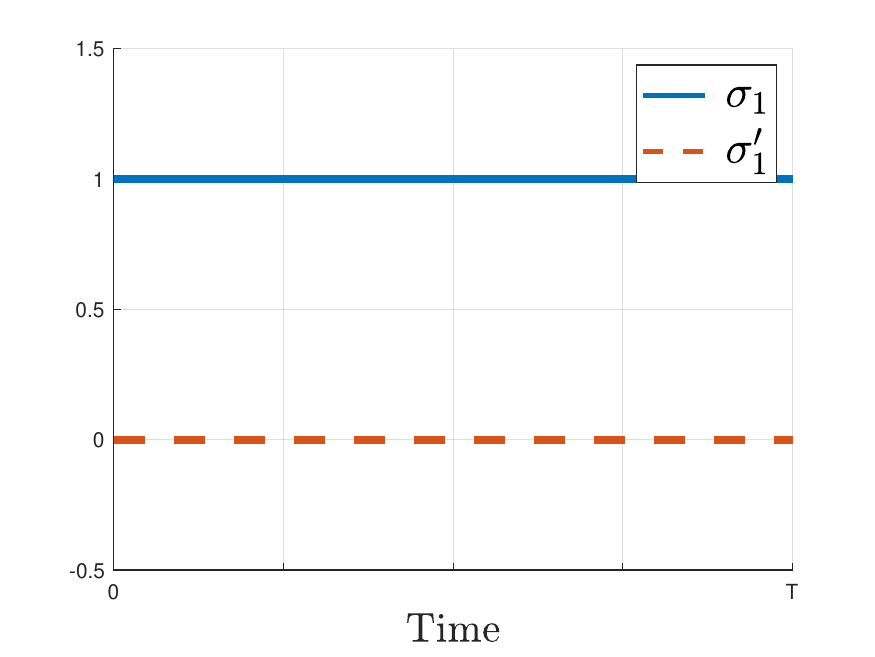}
\includegraphics[scale=0.32]{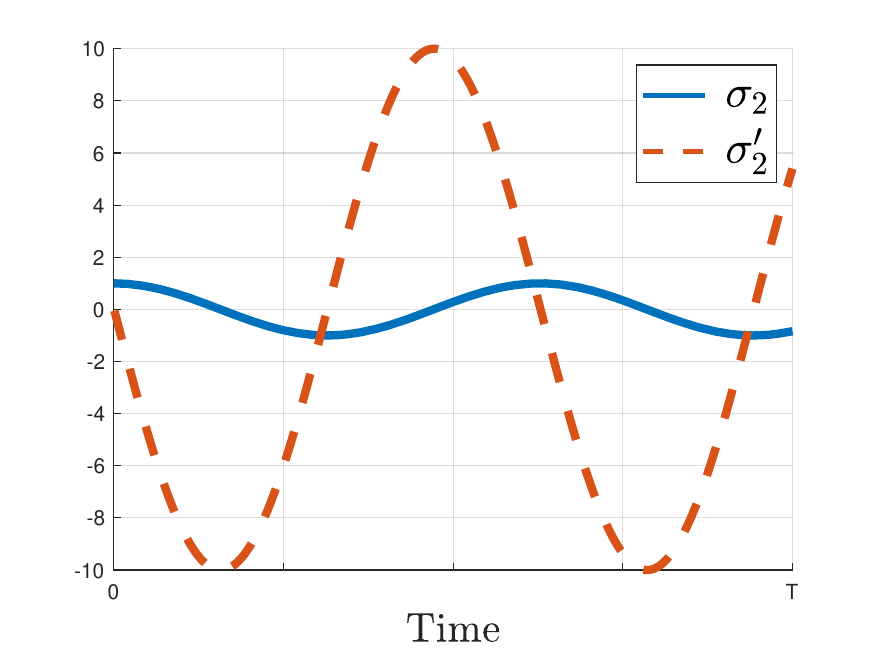}
\includegraphics[scale=0.32]{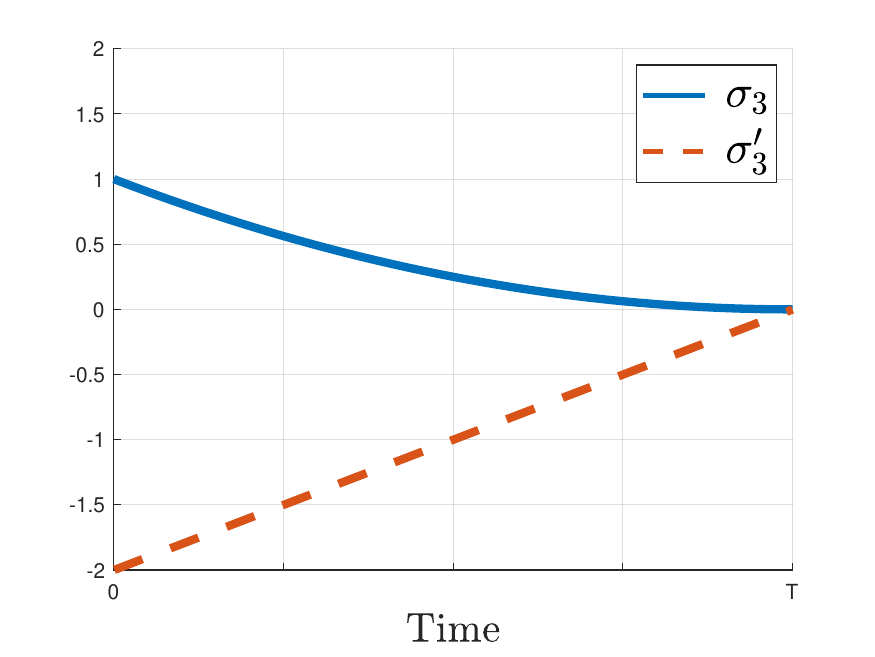}
\end{center}~\\[-10pt]
\caption{Test functions for $\sigma$.}
\label{fig70}
\end{figure}

Clearly, in the case $\sigma=\sigma_3$ we cannot use the formula of Theorem \ref{th1}, therefore, we used that of Theorem \ref{th2}. As shown in Figure \ref{fig70xy},  only in this case there was a noticeable difference in the recovery of $f$.
\begin{figure}[!h]
\begin{center}
{\Large \hspace{0.75em}$\sigma=\sigma_1$\hspace{6.9em}$\sigma=\sigma_2$\hspace{6.9em}$\sigma=\sigma_3$}\\
\includegraphics[scale=0.32]{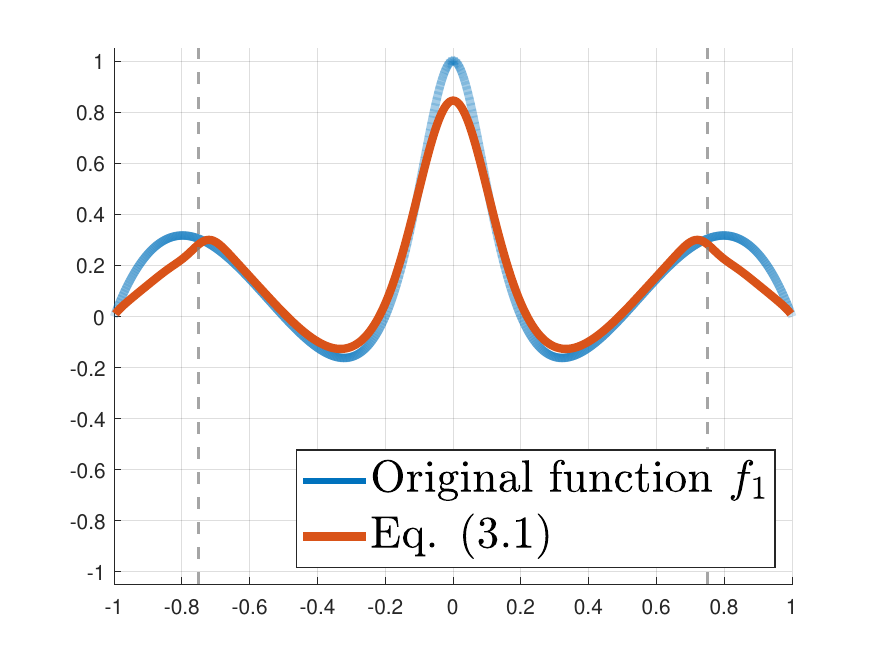}
\includegraphics[scale=0.32]{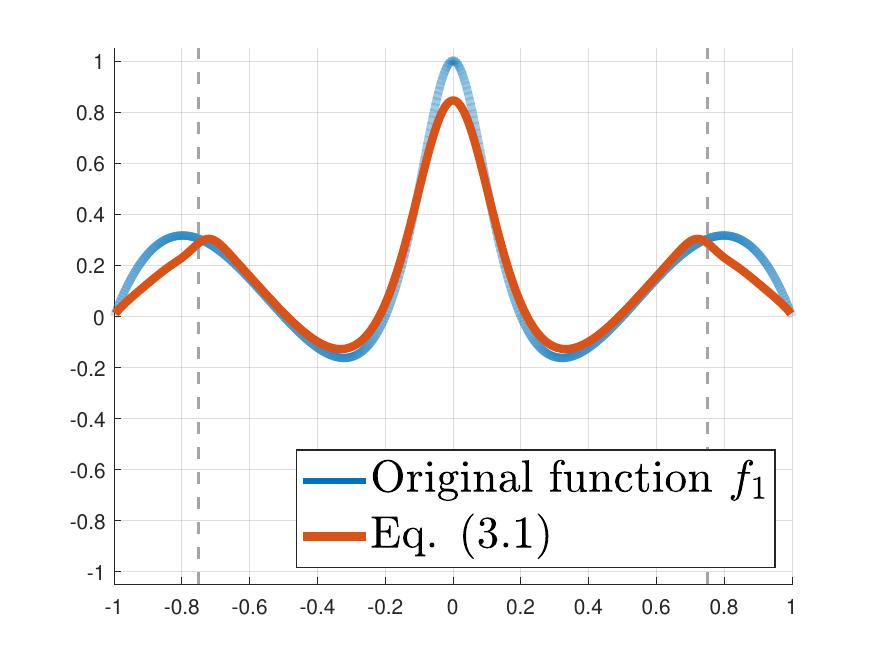}
\includegraphics[scale=0.32]{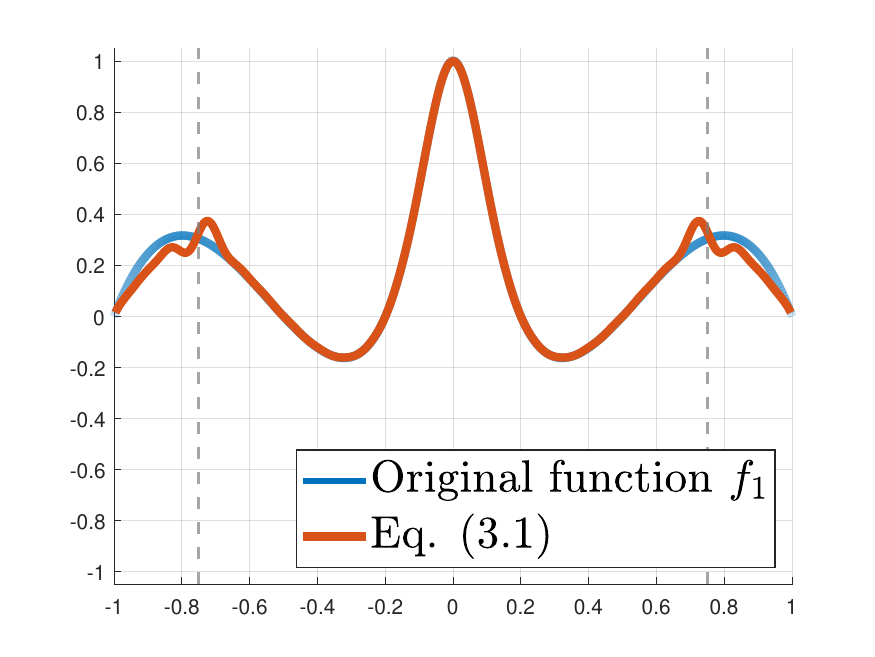}
\includegraphics[scale=0.32]{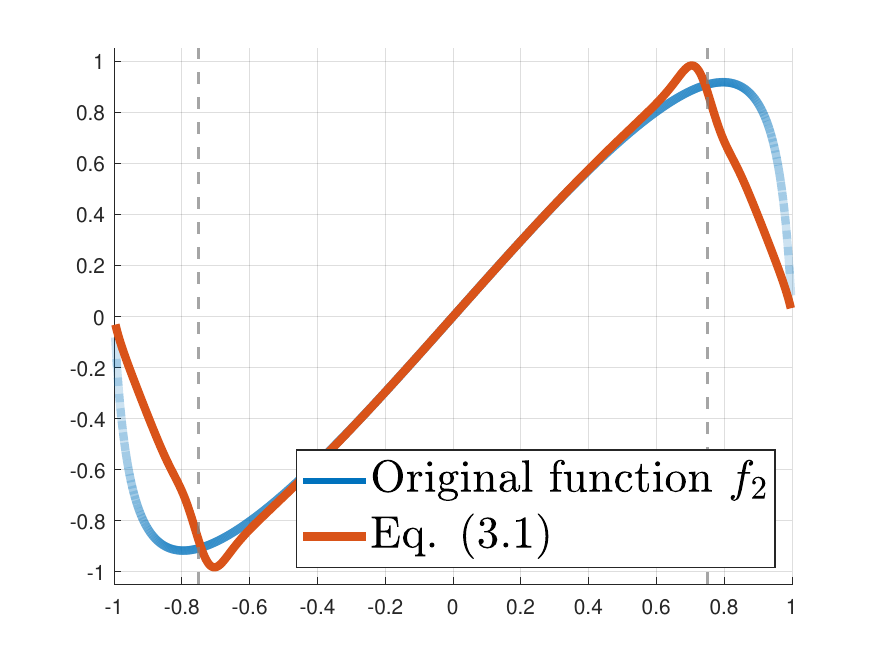}
\includegraphics[scale=0.32]{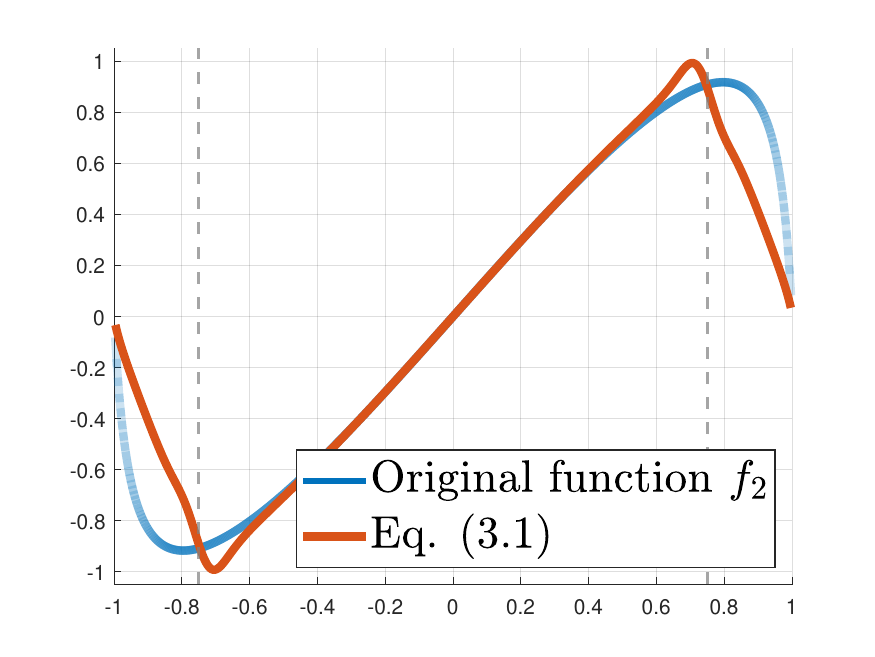}
\includegraphics[scale=0.32]{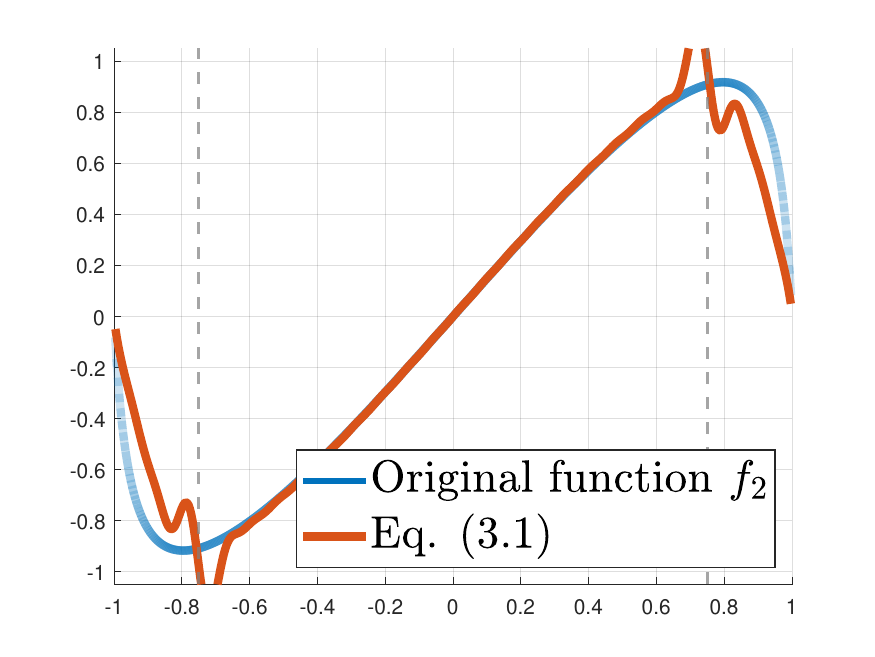}
\includegraphics[scale=0.32]{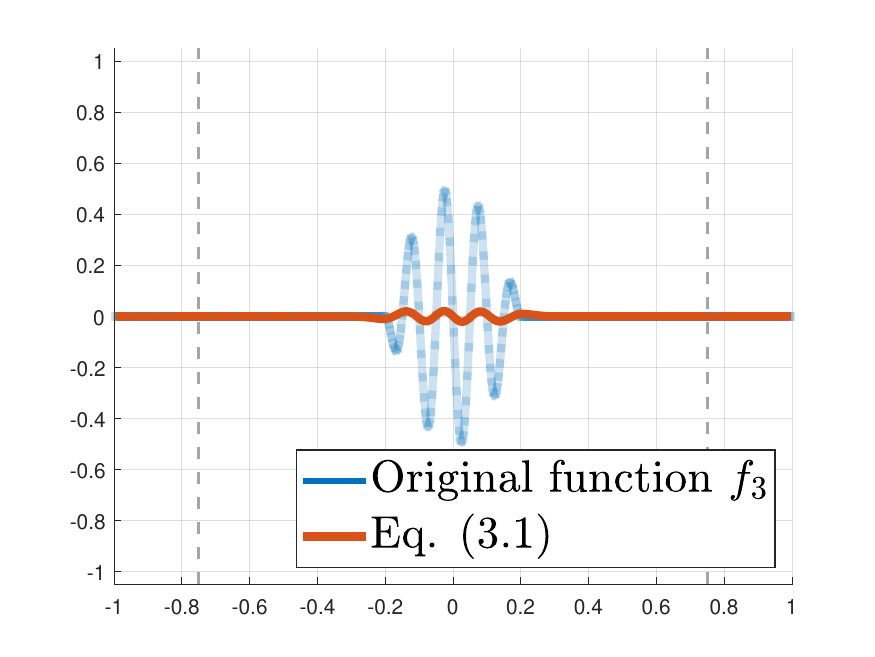}
\includegraphics[scale=0.32]{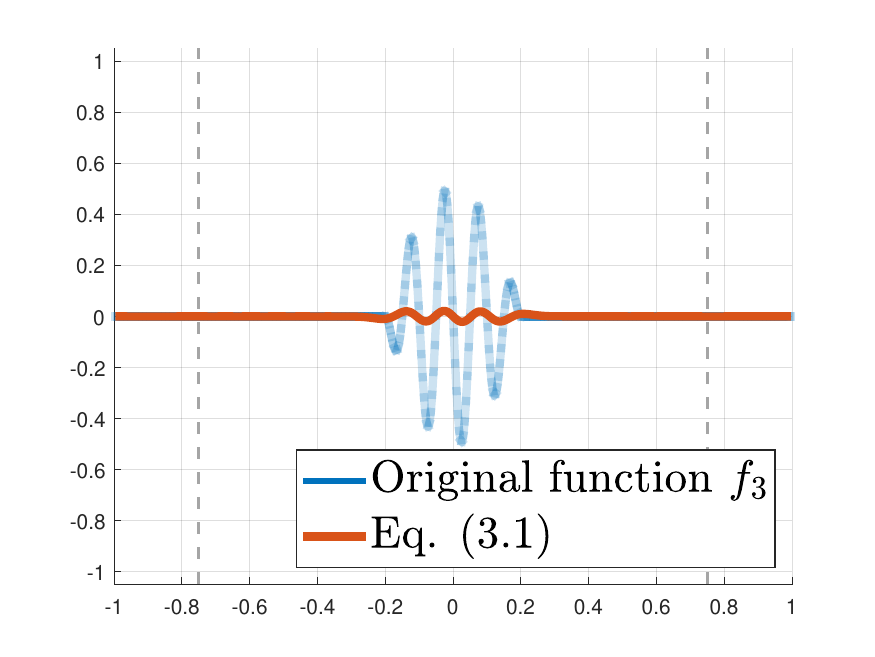}
\includegraphics[scale=0.32]{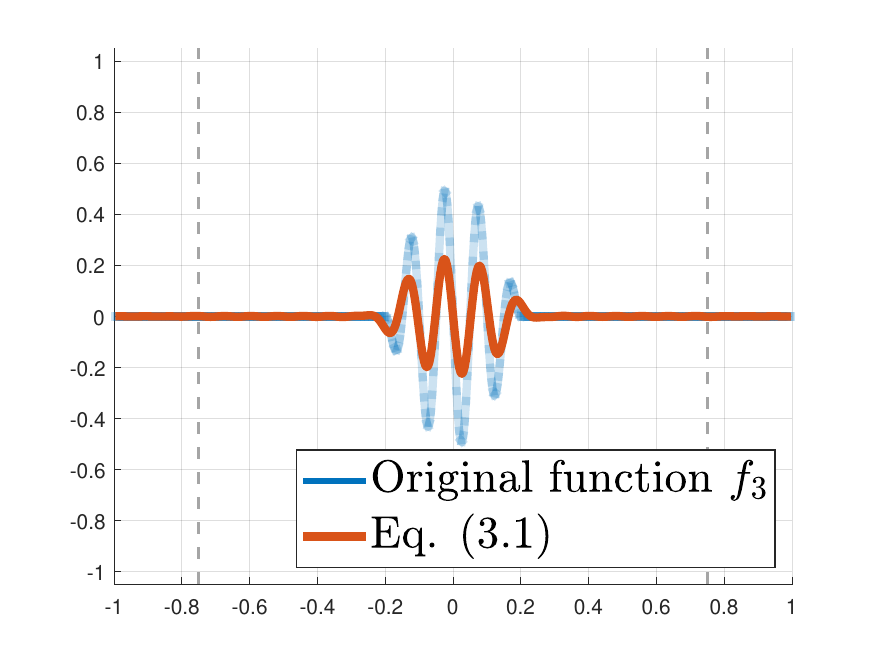}
\includegraphics[scale=0.32]{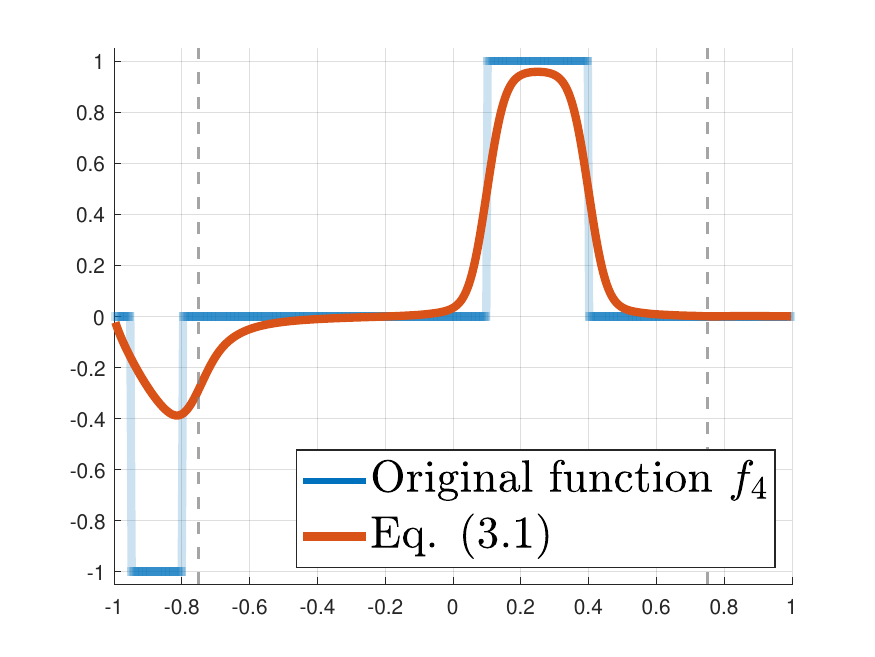}
\includegraphics[scale=0.32]{./Figuras/fig7041.pdf}
\includegraphics[scale=0.32]{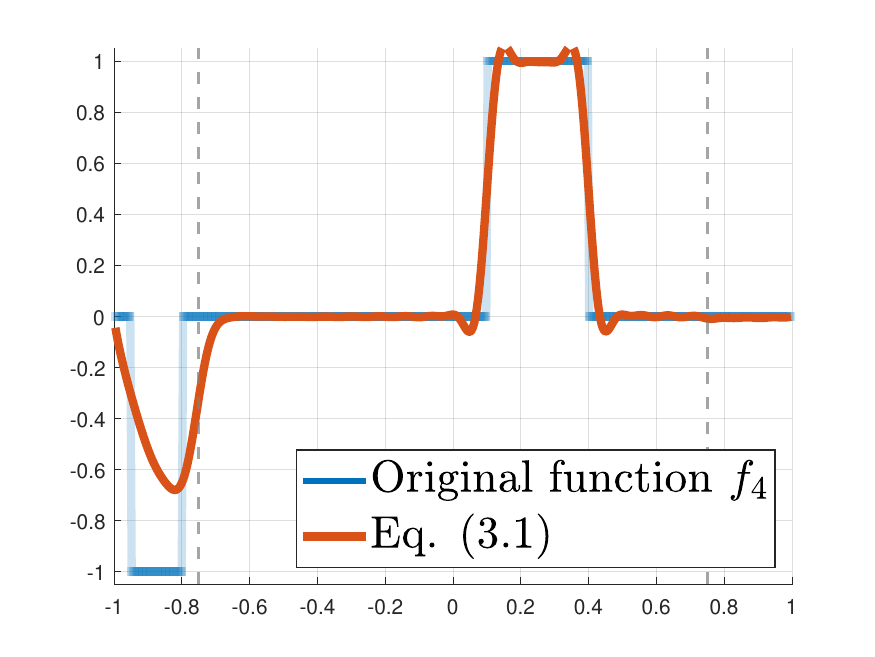}
\end{center}~\\[-10pt]
\caption{Recovery of $f$ for $\sigma_1$, $\sigma_2$ and $\sigma_3$.}
~\\[-10pt]
\label{fig70xy}
\end{figure}

\section{Conclusion}\label{sec-7}
In this study, we addressed the inverse source problem for the fractional heat equation, deriving a reconstruction formula for the spatial component of the source term from partial measurements of the system state and its time derivative. Our approach leverages the null controllability property of the fractional heat equation for $s\in(1/2,1)$, combined with spectral analysis and Volterra integral equations to recover the Fourier coefficients of the unknown source. The proposed methodology was validated through numerical experiments, demonstrating its effectiveness even for discontinuous or high-frequency sources.

Although our method provides a robust framework for source reconstruction, there are avenues for improvement and extension. One key direction is the development of techniques to better capture Fourier coefficients, particularly for high-frequency modes, where numerical errors can accumulate. This could involve adaptive strategies for eigenfunction approximation or refined quadrature rules for associated integral equations. In addition, extending this approach to broader classes of nonlocal operators or time-dependent source terms further enhance its applicability.

Future work could also explore the use of machine learning or data-driven methods to complement the analytical reconstruction formula, potentially improving  the accuracy in cases where the measurements are noisy or limited. Finally, investigating the inverse problem for fractional orders $s\leq 1/2$, where null controllability fails, presents a significant theoretical and computational challenge that warrants further research.

Our results contribute to the growing toolbox for inverse problems in nonlocal PDEs with potential applications in imaging, material science, and other fields. By bridging theoretical insights with practical numerics, this study lays the groundwork for future advances in this field.

\section*{Acknowledgements}
The authors wish to express their gratitude to Mart\'in Hern\'andez for taking time to critically review the manuscript.

Part of this research was conducted  while the third author visiting the FAU DCN-AvH Chair for Dynamics, Control, Machine Learning and Numerics at the Department of Mathematics at FAU, Friedrich- Alexander-Universit\"{a}t Erlangen-N\"{u}rnberg (Germany) with the support of a  Research Fellowship for Experienced Researchers from the  Alexander von Humboldt Foundation.

\appendix
\renewcommand{\thesection}{Appendix \Alph{section}}
\section{Derivation of an Improved Integration Rule}

\label{apA}

In Chapter \ref{sec-5}, we approximated the term  
\begin{equation}
\label{intrule0}
\int_{x_{j-1}}^{x_{j+1}}\phi_j(x)\varphi_m(x)\,dx \approx h\varphi_m(x_j),
\end{equation}
by using the trapezoidal rule. However, in this section, we derive a more refined integration rule of the form  
\begin{equation*}
\int_{x_{j-1}}^{x_{j+1}}\phi_j(x)f(x)\,dx \approx Af(x_{j-1}) + Bf(x_j) + Cf(x_{j+1}),
\end{equation*}
which can then be directly applied to $f(x) = \varphi_n(x)$, $n \geq 1$. This new rule is shown to be exact for polynomials with degree at least two degrees.

To derive coefficients $A$, $B$, and $C$, we ensure the exactness of polynomials of increasing degrees:  

For $f(x) = 1$
\begin{equation*}
A + B + C = \int_{x_{j-1}}^{x_{j+1}} \phi_j(x)\,dx = 2\int_{x_{j-1}}^{x_j} \left(1 - \frac{x_j - x}{h}\right)dx,
\end{equation*}
which leads to 
\begin{equation*}
A + B + C = h.
\end{equation*}

Now, for $f(x) = x$:
\begin{equation*}
Ax_{j-1} + Bx_j + Cx_{j+1} = \int_{x_{j-1}}^{x_{j+1}} \phi_j(x)x\,dx.
\end{equation*}
Evaluating this integral:
\begin{align*}
\int_{x_{j-1}}^{x_{j+1}} \phi_j(x)x\,dx &= \int_{x_{j-1}}^{x_j} \left(x - \frac{x_j x - x^2}{h}\right)dx + \int_{x_j}^{x_{j+1}} \left(x - \frac{x^2 - x_j x}{h}\right)dx.
\end{align*}
Simplifications yield:
\begin{equation*}
A - C = 0,
\end{equation*}
as $x_{j-1} = x_j - h$ and $x_{j+1} = x_j + h$.

In the case where $f(x)=x^2$, on one hand we have:
\begin{align*}
Ax_{j-1}^2 + Bx_j^2 + Cx_{j+1}^2 &= A(x_j - h)^2 + Bx_j^2 + C(x_j + h)^2 \\
&= A(x_j^2 - 2hx_j + h^2) + Bx_j^2 + C(x_j^2 + 2hx_j + h^2) \\
&= x_j^2(A + B + C) - 2hx_j(A - C) + h^2(A + C),
\end{align*}
and since $A + B + C = h$, this simplifies to:
\begin{equation*}
x_j^2h + h^2(A + C).
\end{equation*}
On the other hand, evaluating the integral:
\begin{align*}
\int_{x_{j-1}}^{x_{j+1}} \phi_j(x)x^2\,dx &= \int_{x_{j-1}}^{x_j} \left(x^2 - \frac{x_j x^2 - x^3}{h}\right)dx + \int_{x_j}^{x_{j+1}} \left(x^2 - \frac{x^3 - x_j x^2}{h}\right)dx \\
&= \frac{x_{j+1}^4 - 2x_j^4 + x_{j-1}^4}{12h},
\end{align*}
and substituting $x_{j-1} = x_j - h$ and $x_{j+1} = x_j + h$, we get:
\begin{equation*}
\int_{x_{j-1}}^{x_{j+1}} \phi_j(x)x^2\,dx = \frac{12x_j^2h^2 + 2h^4}{12h}.
\end{equation*}
This implies:
\begin{equation*}
A + C = \frac{h}{6}.
\end{equation*}

By solving the linear system formed by these equations, we conclude that $A = C = \frac{h}{12}$ and  $B = \frac{5h}{6}$.

Thus, an improved approximation for \eqref{intrule0} is:
\begin{equation}
(\phi_j)_n \approx \frac{h}{12} \Big(\varphi_n(x_{j-1}) + 10\varphi_n(x_j) + \varphi_n(x_{j+1})\Big).
\label{intrule3}
\end{equation}

This rule (see \eqref{intrule3}) exploits the structure of $\phi_j$. However, because  we often approximate in the discrete space $\varphi_m \in V_h$, we can leverage its representation in terms of the basis $\{\phi_j\}_{j=1}^N$. Thus, we computed
\begin{align*}
\int_{x_{j-1}}^{x_{j+1}}\phi_j(x)\varphi_m(x)\,dx &= \int_{x_{j-1}}^{x_j} \Big(\varphi_m(x_{j-1})\phi_j(x)\phi_{j-1}(x) + \varphi_m(x_j)|\phi_j(x)|^2\Big)dx \\
&\quad + \int_{x_j}^{x_{j+1}} \Big(\varphi_m(x_j)|\phi_j(x)|^2 + \varphi_m(x_{j+1})\phi_j(x)\phi_{j+1}(x)\Big)dx.
\end{align*}
Simplifications yield:
\begin{equation*}
\int_{x_{j-1}}^{x_{j+1}}\phi_j(x)\varphi_m(x)\,dx = \frac{h}{6} \Big(\varphi_m(x_{j-1}) + 4\varphi_m(x_j) + \varphi_m(x_{j+1})\Big).
\end{equation*}

\section*{Conflict of Interest declaration}
The authors declare that they have no affiliations with or involvement in any organization or entity with any
financial interest in the subject matter or materials discussed in this manuscript.



\end{document}